\documentclass[a4paper,12pt]{article}
\usepackage{ucs}
\usepackage{amsfonts, amssymb, amsmath, amsthm, amscd}
\usepackage{graphicx}
\usepackage{cite}
\ifx\undefined \pdfgentounicode \else
\input{glyphtounicode} \pdfgentounicode=1
\fi

\author{A.A. Vasil'eva\footnote{Faculty of Mechanics and Mathematics, Lomonosov Moscow State University; Moscow Center for Fundamental and Applied Mathematics.}}
\title{Kolmogorov widths of anisotropic Sobolev classes}
\date{}
\begin{document}

\maketitle

\newenvironment{Biblio}{%
                  \renewcommand{\refname}{\footnotesize REFERENCES}%
                  \begin{thebibliography}{99}%
                  \itemsep -0.2em}%
                  {\end{thebibliography}}

\def\inff{\mathop{\smash\inf\vphantom\sup}}
\renewcommand{\le}{\leqslant}
\renewcommand{\ge}{\geqslant}
\newcommand{\sgn}{\mathrm {sgn}\,}
\newcommand{\inter}{\mathrm {int}\,}
\newcommand{\dist}{\mathrm {dist}}
\newcommand{\supp}{\mathrm {supp}\,}
\newcommand{\R}{\mathbb{R}}
\newcommand{\Z}{\mathbb{Z}}
\newcommand{\N}{\mathbb{N}}
\newcommand{\Q}{\mathbb{Q}}
\theoremstyle{plain}
\newtheorem{Trm}{Theorem}
\newtheorem{trma}{Theorem}

\newtheorem{Def}{Definition}
\newtheorem{Cor}{Corollary}
\newtheorem{Lem}{Lemma}
\newtheorem{Rem}{Remark}
\newtheorem{Sta}{Proposition}
\newtheorem{Exa}{Example}
\renewcommand{\proofname}{\bf Proof}
\renewcommand{\thetrma}{\Alph{trma}}

\begin{abstract}
In this article, order estimates for the Kolmogorov widths of periodic anisotropic Sobolev and Nikol'skii classes are obtained, as well as order estimates for the Kolmogorov widths of anisotropic finite-dimensional balls.
\end{abstract}

\section{Introduction}

In this paper, we obtain order estimates for the Kolmogorov widths of anisotropic periodic Sobolev and Nikol'skii classes $W^{\overline{r}}_{\overline{p}, \overline{\alpha}}(\mathbb{T}^d)$ and $H^{\overline{r}}_{\overline{p}}(\mathbb{T}^d)$ in the anisotropic space $L_{\overline{q}}(\mathbb{T}^d)$ (see definitions below). Here $\mathbb{T}= [0, \, 2\pi]$, $\overline{p} = (p_1, \, \dots, \,  p_d)$, $\overline{q} = (q_1, \, \dots, \,  q_d)$, $\overline{r} = (r_1, \, \dots, \,  r_d)$, $\overline{\alpha} = (\alpha_1, \, \dots, \,  \alpha_d)$, $r_j>0$, $\alpha_j\in \R$. The following cases will be considered: 1) $1\le p_j\le \infty$, $2\le q_j<\infty$, $1\le j\le d$; 2) $1\le p_j\le q_j\le 2$ for $1\le j\le \nu$, $p_j\ge q_j$ for $\nu+1\le j\le d$, where $\nu\in \{0, \, 1, \, \dots, \, d\}$. Hence we generalize well-known results to anisotropic $\overline{p}$ and $\overline{q}$ under the conditions on the parameters mentioned above. In addition, we obtain order estimates for the Kolmogorov $n$-widths of finite-dimensional balls $B_{p_1, \, \dots, \, p_d}^{k_1, \, \dots, \, k_d}$ in $l_{q_1, \, \dots, \, q_d}^{k_1, \, \dots, \, k_d}$ for $1\le p_j\le \infty$, $2\le q_j<\infty$, $1\le j\le d$, $n\le \frac{k_1\dots k_d}{2}$. For $d=1$, this result was obtained by Gluskin \cite{bib_gluskin}, and for $d=2$, it was obtained in \cite{vas_besov}.

We need some notation.

Let $1\le p_1,\, \dots, \, p_d\le\infty$, $\overline{p} = (p_1, \, \dots, \,  p_d)$, $(X_1, \, \mu_1), \, \dots, \, (X_d, \, \mu_d)$ be measure spaces. The anisotropic norm $\|\cdot \|_{L_{\overline{p}}(X_1\times \dots \times X_d, \, \mu_1\otimes \dots \otimes\mu_d)}$ is defined by induction on $d$: for $d=1$, it is the standard norm in the space $L_{p_1}(X_1, \, \mu_1)$, and for $d\ge 2$, the recurrence relation is as follows:
\begin{align}
\label{anis_norm_def}
\begin{array}{c}
\|f\| _{L_{\overline{p}}(X_1\times \dots \times X_d, \, \mu_1\otimes \dots \otimes\mu_d)} = \|g_f\|_{L_{p_d}(X_d, \, \mu_d)}, \\
g_f(x_d) = \|f(\cdot, \, x_d)\|_{L_{(p_1, \, \dots, \, p_{d-1})}(X_1\times \dots \times X_{d-1,} \, \mu_1\otimes \dots \otimes\mu_{d-1})};
\end{array}
\end{align}
here $f$ is a measurable function on $(X_1\times \dots \times X_d, \, \mu_1\otimes \dots \otimes\mu_d)$.
In particular, for finite $p_1, \, \dots, \, p_d$,
$$
\|f\|_{L_{\overline{p}}(X_1\times \dots \times X_d, \, \mu_1\otimes \dots \otimes\mu_d)} = $$$$=\Bigl(\int \limits _{X_d}\dots\Bigl(\int \limits _{X_2}\Bigl( \int \limits _{X_1}|f(x_1, \, x_2, \, \dots, \, x_d)|^{p_1}\, d\mu_1\Bigr)^{p_2/p_1}\, d\mu_2\Bigr)^{p_3/p_2}\dots \, d\mu_d\Bigr)^{1/p_d}.
$$
The space $L_{\overline{p}}(X_1\times \dots \times X_d, \, \mu_1\otimes \dots \otimes\mu_d)$ consists of equivalence classes of measurable functions such that $\|f\|_{L_{\overline{p}}(X_1\times \dots \times X_d, \, \mu_1\otimes \dots \otimes\mu_d)}<\infty$.

We denote by $\mathbb{T}$ the interval $[0, \, 2\pi]$ with ``glued'' endpoints, equipped with the normalized Lebesgue measure (i.e., the measure of $[a, \, b]\subset \mathbb{T}$ is $(b-a)/2\pi$). 
If we take as $(X_j, \, \mu_j)$ the set $\mathbb{T}$ with normalized Lebesgue measure for each $j=1, \, \dots, \, d$, then the corresponding anisotropic norm will be denoted by $\|\cdot\|_{L_{\overline{p}}(\mathbb{T}^d)}$.

More information about the spaces with anisotropic norm can be found, e.g., in \cite{bed_pan, gal_pan}.

Let us define the anisotropic Sobolev and Nikol'skii spaces according to \cite{teml_book}.

Let $r>0$, $\alpha \in \R$. The Bernoulli kernel $F_r(\cdot, \, \alpha)$ is defined as follows:
\begin{align}
\label{frxa}
F_r(x, \, \alpha) = 1 + 2\sum \limits _{k=1}^\infty k^{-r}\cos (kx - \alpha \pi/2), \quad x \in \mathbb{R}.
\end{align}
It is well-known \cite[Theorem 1.4.1]{teml_book} that $F_r(\cdot, \, \alpha) \in L_1(\mathbb{T})$, and the series \eqref{frxa} converges to $F_r(\cdot, \, \alpha)$ in $L_1(\mathbb{T})$.

Let $\overline{r} = (r_1, \, \dots, \, r_d)$, $\overline{\alpha} = (\alpha_1, \, \dots, \, \alpha_d)$, $\overline{p}=(p_1, \, \dots, \, p_d)$, $r_j>0$, $\alpha_j\in \R$, $1\le p_j\le \infty$, $j=1, \, \dots, \, d$. 

\begin{Def}
The class $W^{\overline{r}}_{\overline{p}, \overline{\alpha}}(\mathbb{T}^d)$ consists of functions $f$ on $\mathbb{T}^d$ such that for all $j\in \{1, \, \dots, \, d\}$ the integral representation  
\begin{align}
\label{oooo}
f(x_1, \, \dots, \, x_d) = \frac{1}{2\pi} \int \limits _{\mathbb{T}} \varphi_j(x_1, \, \dots, \, x_{j-1}, \, y, \, x_{j+1}, \, \dots, \, x_d) F_{r_j}(x_j-y, \, \alpha_j)\, dy
\end{align}
holds with $\|\varphi_j\|_{L_{\overline{p}}(\mathbb{T}^d)}\le 1$.
\end{Def}

Let $h\in \R$, $f\in L_{\overline{p}}(\mathbb{T}^d)$. We continue $f$ periodically to $\R^d$ and set $$\Delta_h^{1,j} f(x_1, \, \dots, \, x_d) = f(x_1, \, \dots, \, x_{j-1}, \, x_j+h, \, x_{j+1}, \, \dots, \, x_d) - $$$$-f(x_1, \, \dots, \, x_{j-1}, \, x_j, \, x_{j+1}, \, \dots, \, x_d).$$ For $l\in \N$, $l\ge 2$, the operator $\Delta_h^{l,j}$ is defined by the equation $\Delta_h^{l,j} = \Delta_h^{1,j}\circ \Delta_h^{l-1,j}$.

\begin{Def}
The Nikol’skii class $H^{\overline{r}}_{\overline{p}}(\mathbb{T}^d)$ consists of functions $f\in L_{\overline{p}}(\mathbb{T}^d)$ such that
$$
\|f\|_{L_{\overline{p}}(\mathbb{T}^d)}\le 1, \quad \|\Delta^{l_j,j}_hf\|_{L_{\overline{p}}(\mathbb{T}^d)} \le |h|^{r_j}, \quad h\in \R, \quad 1\le j\le d,
$$
where $l_j =\lfloor r_j\rfloor +1$.
\end{Def}

It is well-known \cite[Theorem 3.4.6]{teml_book} that
\begin{align}
\label{emb_w_h} W^{\overline{r}}_{\overline{p}, \overline{\alpha}}(\mathbb{T}^d) \subset C(\overline{r}) H^{\overline{r}}_{\overline{p}}(\mathbb{T}^d),
\end{align}
where $C(\overline{r})$ is a positive number depending only on $\overline{r}$.

Recall the definition of the Kolmogorov widths (for details, see \cite{kniga_pinkusa, alimov_tsarkov, nvtp, teml_book, itogi_nt}).
\begin{Def}
Let $X$ be a normed space, and let $M\subset X$, $n\in
\Z_+$. The Kolmogorov $n$-width of $M$ in $X$ is defined by
$$
d_n(M, \, X) = \inf _{L\in {\cal L}_n(X)} \sup _{x\in M} \inf
_{y\in L} \|x-y\|;
$$
here ${\cal L}_n(X)$ is the family of all subspaces in $X$
of dimension at most $n$.
\end{Def}

In 1960--80's, the problem of estimating the Kolmogorov widths of the balls $B_p^N$ in $l_q^N$ (see notation below) and Sobolev classes in $L_q$ was studied \cite{pietsch1, stesin, kashin_oct, bib_kashin, gluskin1, bib_gluskin, garn_glus, itogi_nt, bib_ismag, vmt60, kashin_sma}. For periodic Sobolev classes defined by restrictions on one or several partial derivatives, as well as for Nikol'skii classes, this problem was studied by Temlyakov \cite{teml1, teml2, teml3, teml4, teml5} and Galeev \cite{galeev85, galeev87}; see also \cite{galeev4, itogi_nt, teml_book}. In \cite{galeev85}, anisotropic norms were considered, but only for the following cases: 
\begin{align}
\label{emg85cases}
\begin{array}{c}
1)\; 1<q_j\le p_j<\infty, \; 1\le j\le d, \; 2)\; 1<p_j\le q_j\le 2, \; 1\le j\le d, \\ 3)\; 2\le p_j\le q_j<\infty, \; 1\le j\le d, \\ 4)\; 1<p_j\le 2\le q_j<\infty,\; 1\le j\le d, \; p_1=\dots=p_d.
\end{array}
\end{align}
(Actually, from the
proof we can see that $p_1=\dots=p_d$, $q_1=\dots=q_d$ in case 3), i.e., the norms are isotropic.)
In cases 3), 4), the estimates of the widths were obtained only for the ``large smoothness''. The case of ``small smoothness'' was studied in \cite{kashin_sma} for univariate functions and in \cite{galeev87} for multivariate functions and isotropic norms. 

Notice that in \cite{galeev85, galeev87, galeev4} the estimates for the widths were obtained for $1<p_j<\infty$, $1<q_j<\infty$; the discretization was based on Littlewood--Paley theorem and Marcinkiewicz multiplier theorem \cite{besov_iljin_nik, nikolski_sm, besov_littlewood}, as well as on Marcinkiewicz--Zygmund theorem (see \cite[Chapter X, Theorem 7.5]{zigmund} for the univariate case, and \cite[Theorem 1.2]{galeev85}, for the multivariate case). In \cite{teml3, teml4, teml_book}, an alternative approach was applied, and the estimates for the widths were obtained for $1\le p, \, q\le \infty$ (here $p_1=\dots=p_d=p$, $q_1=\dots=q_d=q$); the main idea was to replace the Dirichlet kernels by de la Vall\'ee Poussin kernels.

In \cite{galeev1, galeev2}, estimates for the widths of intersections of Sobolev classes in $L_q(\mathbb{T})$ were obtained (for $q>2$, only the case of ``large smoothness'' was considered); this result was generalized to the case of ``small smoothness'' in \cite{vas_int_sob}.

In \cite{akishev, akishev1, akishev2, akishev3}, the problem of estimating the widths and the best $M$-term approximations of anisotropic Besov--Nikol'skii--Amanov classes in an anisotropic Lorentz space was studied; the parameters $p_j$, $q_j$ in \cite{akishev, akishev3} satisfied \eqref{emg85cases}.

In this paper, we obtain order estimates for the widths $d_n(W^{\overline{r}}_{\overline{p},\overline{\alpha}}(\mathbb{T}^d), \, L_{\overline{q}}(\mathbb{T}^d))$ and $d_n(H^{\overline{r}}_{\overline{p}}(\mathbb{T}^d), \, L_{\overline{q}}(\mathbb{T}^d))$. The following cases will be considered: 1) $2\le q_j<\infty$, $1\le p_j\le \infty$, $1\le j\le d$; 2) $1\le p_j\le q_j\le 2$, $1\le j\le \nu$, $p_j\ge q_j$, $\nu+1\le j\le d$.

We need some more notation.

Given $2\le q<\infty$, $1\le p\le \infty$, we set 
\begin{align}
\label{om_pq}
\omega_{p,q} = \begin{cases} 0 & \text{for }p> q, \\ \frac{1/p-1/q}{1/2-1/q} & \text{for }2<p\le q, \\ 1 & \text{for }1\le p\le 2.\end{cases} 
\end{align}

Let $I\subset \{1, \, \dots, \, d\}$ be a non-empty set, $\overline{p}=(p_1, \, \dots, \, p_d)$. We define the number $\langle \overline{p}\rangle_{I}$ by
\begin{align}
\label{sredn_1}
\frac{1}{\langle \overline{p}\rangle_{I}} = \frac{1}{|I|} \sum \limits _{j\in I} \frac{1}{p_j}.
\end{align}
We write 
\begin{align}
\label{sredn_2}
\langle \overline{p}\rangle:= \langle \overline{p}\rangle_{\{1, \, \dots, \, d\}}.
\end{align}
For $I=\varnothing$, we set $\langle \overline{p}\rangle_{I}=1$.

Let $\sigma$ be a permutation of $d$ elements such that \begin{align}
\label{upor} \omega_{p_{\sigma(1)},q_{\sigma(1)}} \le \omega_{p_{\sigma(2)},q_{\sigma(2)}}\le \dots \le \omega_{p_{\sigma(d)},q_{\sigma(d)}}.
\end{align}
The numbers $\mu$, $\nu\in \{0, \, \dots, \, d\}$ are defined by 
\begin{align}
\label{mu_nu_def} \{1, \, \dots, \, \mu\} = \{j:\; \omega_{p_{\sigma(j)},q_{\sigma(j)}} = 0\}, \; \{1, \, \dots, \, \nu\} = \{j:\; \omega_{p_{\sigma(j)},q_{\sigma(j)}}<1\}.
\end{align}
From \eqref{om_pq} it follows that the condition $j\in \{1, \, \dots, \, \mu\}$ is equivalent to $p_{\sigma(j)}\ge q_{\sigma(j)}$ for $q_{\sigma(j)}>2$, and $p_{\sigma(j)}> q_{\sigma(j)}$, for $q_{\sigma(j)}=2$; the condition $j\in \{1, \, \dots, \, \nu\}$ is equivalent to $p_{\sigma(j)}> 2$.

We set $$I(t, \, s) = \{\sigma(t), \, \sigma(t+1),\, \dots, \, \sigma(s-1), \, \sigma(s)\}, \quad 1\le t\le s\le d.$$

Given $\overline{a}=(a_1, \, \dots, \, a_d)$, $\overline{b}=(b_1, \, \dots, \, b_d)$, we write $\overline{a}\circ \overline{b}=(a_1b_1, \, \dots, \, a_db_d)$.

Let $X$, $Y$ be sets, $f_1$, $f_2:\ X\times Y\rightarrow \mathbb{R}_+$.
We denote $f_1(x, \, y)\underset{y}{\lesssim} f_2(x, \, y)$ (or $f_2(x, \, y)\underset{y}{\gtrsim} f_1(x, \, y)$) if for each $y\in Y$ there is $c(y)>0$ such that $f_1(x, \, y)\le
c(y)f_2(x, \, y)$ for all $x\in X$; $f_1(x, \,
y)\underset{y}{\asymp} f_2(x, \, y)$ if $f_1(x, \, y)
\underset{y}{\lesssim} f_2(x, \, y)$ and $f_2(x, \,
y)\underset{y}{\lesssim} f_1(x, \, y)$.

\begin{Trm}
\label{main1}
Let $d\in \N$, $r_j>0$, $\alpha_j\in \R$, $1\le p_j\le\infty$, $2\le q_j<\infty$, $j=1, \, \dots, \, d$. Suppose that 
\begin{align}
\label{emb_cond}
1 + \frac{d-\mu}{\langle \overline{r}\circ \overline{q} \rangle _{I(\mu+1, \, d)}} - \frac{d-\mu}{\langle \overline{r}\circ \overline{p} \rangle _{I(\mu+1, \, d)}}
>0.
\end{align}
We set
$$
\theta_t =\frac{1}{\frac{t}{\langle \overline{r} \rangle _{I(1, \, t)}}+\frac{2(d-t)}{\langle \overline{r}\circ\overline{q}\rangle _{I(t+1,\, d)}}}\left(1+(d-t)\left(\frac{1}{\langle \overline{r}\circ\overline{q}\rangle_{I(t+1,\, d)}}-\frac{1}{\langle \overline{r}\circ\overline{p}\rangle_{I(t+1,\, d)}}\right)\right),
$$
$\mu\le t\le \nu$; if $\nu<d$ and there is $j\in \{\nu+1, \, \dots, \, d\}$ such that $q_{\sigma(j)}>2$, we also write
$$
\theta_d = \frac{\langle\overline{r}\rangle}{d}\Bigl(1+\frac{d-\nu}{2\langle \overline{r}\rangle_{I(\nu+1,\, d)}}-\frac{d-\nu}{\langle \overline{r} \circ\overline{p}\rangle_{I(\nu+1,\, d)}}\Bigr).
$$
Let $J=\{\mu, \, \mu+1, \, \dots, \, \nu\}\cup \{d\}$ if $\nu<d$ and there is $j\in \{\nu+1, \, \dots, \, d\}$ such that $q_{\sigma(j)}>2$; otherwise, we set $J=\{\mu, \, \mu+1, \, \dots, \, \nu\}$.
Suppose that there is $j_*\in J$ such that 
\begin{align}
\label{theta_j_min}
\theta_{j_*}<\min _{j\in J\backslash\{j_*\}}\theta_j. 
\end{align}
Then
$$
d_n(W^{\overline{r}}_{\overline{p}, \overline{\alpha}}(\mathbb{T}^d), \, L_{\overline{q}}(\mathbb{T}^d)) \underset{\overline{p}, \overline{q}, \overline{r},d}{\asymp} d_n(H^{\overline{r}}_{\overline{p}}(\mathbb{T}^d), \, L_{\overline{q}}(\mathbb{T}^d)) \underset{\overline{p}, \overline{q}, \overline{r},d}{\asymp}n^{-\theta_{j_*}}.
$$
\end{Trm}

\begin{Trm}
\label{main2} Let $r_j>0$, $\alpha_j\in \R$, $1\le j\le d$, $\nu\in \{0, \, \dots, \, d\}$, $1\le p_j\le q_j\le 2$ for $1\le j\le \nu$, and $1\le q_j\le p_j\le \infty$, for $\nu+1\le j\le d$. Suppose that
\begin{align}
\label{theta_posit}
\theta:=\frac{\langle\overline{r}\rangle}{d}\left(1+\frac{\nu}{\langle \overline{r} \circ \overline{q}\rangle_{\{1, \, \dots, \, \nu\}}} - \frac{\nu}{\langle \overline{r} \circ \overline{p}\rangle_{\{1, \, \dots, \, \nu\}}}\right)>0.
\end{align}
Then
$$
d_n(W^{\overline{r}}_{\overline{p},\overline{\alpha}}(\mathbb{T}^d), \, L_{\overline{q}}(\mathbb{T}^d))
\underset{\overline{p}, \overline{q}, \overline{r},d}{\asymp} d_n(H^{\overline{r}}_{\overline{p}}(\mathbb{T}^d), \, L_{\overline{q}}(\mathbb{T}^d)) \underset{\overline{p}, \overline{q}, \overline{r}, d}{\asymp} n^{-\theta}.
$$
\end{Trm}
Notice that the cases $\nu=0$ and $\nu=d$ in Theorem \ref{main2} are well-known; see, e.g., \cite{galeev85}.

In order to prove Theorems \ref{main1} and \ref{main2}, we use discretization method from \cite{teml_book}. The problem will be reduced to estimating the widths of finite-dimensional balls in the anisotropic norm.

Given $N\in \N$, $1\le s\le \infty$, $(x_i)_{i=1}^N\in \mathbb{R}^N$, we write $\|(x_i)_{i=1}^N\|_{l_s^N} = \left(\sum \limits _{i=1}^N |x_i|^s\right)^{1/s}$ for $s<\infty$, and $\|(x_i)_{i=1}^N\|_{l_s^N} = \max _{1\le i\le N}|x_i|$, for $s=\infty$. The space $\mathbb{R}^N$ with this norm is denoted by $l_s^N$; by $B_s^N$, we denote the unit ball in $l_s^N$.

Let $k_1, \, \dots, \, k_d\in \N$, $1\le p_1, \, \dots, \, p_d\le \infty$. By $l_{p_1,\dots, \, p_d}^{k_1,\dots, k_d}$ we denote the space $\mathbb{R}^{k_1\dots k_d}= \{(x_{j_1,\dots,j_d})_{1\le j_s\le k_s, \, 1\le s\le d}:\; x_{j_1,\dots,j_d}\in \mathbb{R}\}$ with norm defined by induction: for $d=1$ it is $\|\cdot\|_{l_{p_1}^{k_1}}$; for $d\ge 2$,
$$
\|(x_{j_1,\dots,j_d})_{1\le j_s\le k_s, \, 1\le s\le d}\|_{l_{p_1,\dots,p_d}^{k_1,\dots,k_d}} = \left\|\bigl(\|(x_{j_1,\dots, \, j_{d-1}, \, j_d})_{1\le j_s\le k_s, \, 1\le s\le d-1}\|_{l_{p_1,\dots, \, p_{d-1}}^{k_1,\dots, \, k_{d-1}}}\bigr)_{j_d=1}^{k_d}\right\|_{l_{p_d}^{k_d}}.
$$
Notice that $$l_{p_1,\dots, \, p_d}^{k_1,\dots, k_d}=L_{\overline{p}}(X_1\times \dots \times X_d, \, \mu_1\otimes \dots \otimes \mu_d),$$ where $X_j=\{1, \, \dots, \, k_j\}$, $\mu_j(\{t\})=1$, $1\le t\le k_j$, $j=1, \, \dots, \, d$.

By $B_{p_1,\dots, \, p_d}^{k_1,\dots, k_d}$ we denote the unit ball of the space $l_{p_1,\dots,p_d}^{k_1,\dots,k_d}$.

For $d=1$ the estimates of the widths of these balls were obtained in \cite{pietsch1, stesin, kashin_oct, bib_kashin, gluskin1, bib_gluskin, garn_glus}. The case $d=2$ was studied in \cite{galeev2, galeev5, izaak1, izaak2, mal_rjut, vas_besov, dir_ull}; for details, see, e.g, \cite{vas_mix2}.

We also notice that in \cite{hinr_mic} the problem of estimating the Kolmogorov and Gelfand widths of balls in finite-dimensional symmetric spaces (including Lorentz and Orlicz spaces) was studied; in \cite{schatten1, schatten2} order estimates for the Kolmogorov, Gelfand and approximation numbers
of embeddings of Schatten classes of $N\times N$-matrices were obtained.

\begin{Trm}
\label{fim_dim_poper}
Let $d\in \N$, $k_1, \, \dots, \, k_d\in \N$, $n\in \Z_+$, $n\le \frac{k_1\dots k_d}{2}$, $2\le q_j<\infty$, $1\le p_j\le \infty$, $j=1, \, \dots, \, d$. Let $\sigma$ be a permutation of $\{1, \, \dots, \, d\}$ such that \eqref{upor} holds.
The numbers $\mu\in \{0, \, \dots, \, d\}$ and $\nu\in \{0, \, \dots, \, d\}$ are defined by \eqref{mu_nu_def}.
Denote by $p_j^* = \max\{p_j, \, 2\}$, $1\le j\le d$.
Then
\begin{align}
\label{fin_dim_est1}
\begin{array}{c}
d_n(B^{k_1, \dots, k_d}_{p_1,\dots,p_d}, \, l^{k_1, \dots, k_d}_{q_1,\dots,q_d}) \underset{\overline{q}}{\asymp} \Phi_{\overline{p},\overline{q}}(k_1, \, \dots, \, k_d, \, n):=\\:=
 \prod _{j=1}^\mu k_{\sigma(j)}^{1/q_{\sigma(j)} - 1/p_{\sigma(j)}} \cdot \min \Bigl\{ 1, \, \min _{\mu+1\le t\le d} \prod _{j=\mu+1}^{t-1} k_{\sigma(j)}^{1/q_{\sigma(j)}-1/p^*_{\sigma(j)}}\times \\ \times(n^{-1/2}k_{\sigma(1)}^{1/2}\dots k_{\sigma(t-1)}^{1/2} k_{\sigma(t)}^{1/q_{\sigma(t)}}\dots k_{\sigma(d)}^{1/q_{\sigma(d)}})^{\omega _{p_{\sigma(t)},q_{\sigma(t)}}}\Bigl\};
\end{array}
\end{align}
if $\nu<d$,
\begin{align}
\label{fin_dim_est2}
\begin{array}{c}
\Phi_{\overline{p},\overline{q}}(k_1, \, \dots, \, k_d, \, n)= \prod _{j=1}^\mu k_{\sigma(j)}^{1/q_{\sigma(j)} - 1/p_{\sigma(j)}} \cdot \min \Bigl\{ 1, \, \min _{\mu+1\le t\le \nu} \prod _{j=\mu+1}^{t-1} k_{\sigma(j)}^{1/q_{\sigma(j)}-1/p_{\sigma(j)}}\times \\ \times(n^{-1/2}k_{\sigma(1)}^{1/2}\dots k_{\sigma(t-1)}^{1/2} k_{\sigma(t)}^{1/q_{\sigma(t)}}\dots k_{\sigma(d)}^{1/q_{\sigma(d)}})^{\omega _{p_{\sigma(t)},q_{\sigma(t)}}}, \\ \prod _{j=\mu+1}^{\nu} k_{\sigma(j)}^{1/q_{\sigma(j)}-1/p_{\sigma(j)}} \cdot n^{-1/2}k_{\sigma(1)}^{1/2}\dots k_{\sigma(\nu)}^{1/2}k_{\sigma(\nu+1)}^{1/q_{\sigma(\nu+1)}}\dots k_{\sigma(d)}^{1/q_{\sigma(d)}}\Bigl\}.
\end{array}
\end{align}
\end{Trm}

The proof generalizes the arguments in \cite{gluskin1, vas_besov}.

\begin{Sta}
\label{fim_dim_poper2} Let $\nu\in \{0, \, \dots, \, d\}$, $1\le p_j\le q_j\le 2$ for $1\le j\le \nu$, $1\le q_j\le p_j\le \infty$ for $\nu+1\le j\le d$, $n \le \frac{k_1\dots k_d}{2}$. Then
$$d_n(B^{k_1, \dots, k_d}_{p_1,\dots,p_d}, \, l^{k_1, \dots, k_d}_{q_1,\dots,q_d}) \asymp k_{\nu+1}^{1/q_{\nu+1}-1/p_{\nu+1}} \dots k_d^{1/q_d-1/p_d}.$$
\end{Sta}

This estimate can be easily derived from Malykhin’s and Rjutin’s result \cite{mal_rjut} on estimates of the widths of a product of multi-dimensional octahedra.

The paper is organized as follows. In \S 2 we prove Theorem \ref{fim_dim_poper} and Proposition \ref{fim_dim_poper2}. In \S 3 we describe the discretization method; here we follow the book \cite{teml_book}. In \S 4 we prove Theorems \ref{main1} and \ref{main2}.

Notice that in \cite[Chapter III, Section 15.2]{besov_iljin_nik}, \cite[Chapter 2 \S 2]{itogi_nt}, \cite{galeev_emb, galeev1, galeev2, galeev85, galeev87, galeev4} the periodic Sobolev classes were defined as sets of functions with zero means in each variable and with restrictions on one or several partial Weyl derivatives. We consider these classes in the end of \S 4.

\section{Estimates for the widths of finite-dimensional balls in an anisotropic norm}

E.D. Gluskin \cite{gluskin1, bib_gluskin} obtained order estimates of $d_n(B_p^N, \, l_q^N)$ for $1\le p< q<\infty$. We formulate this result for $q\ge 2$, $p=2$, $n\le N/2$ (the case $q=p=2$ is obvious):

\begin{trma}
\label{glus} {\rm (see \cite{gluskin1, bib_gluskin}).} Let $2\le q<\infty$, $0\le n\le N/2$.
Then
$$
d_n(B_2^N, \, l_q^N) \underset{q}{\asymp} \min \{1, \,
n^{-1/2}N^{1/q}\}.
$$
\end{trma}

We also need the following assertion.
\begin{trma}
\label{interp_ineq}
{\rm (see \cite{gluskin1}).} Let $0\le \theta\le 1$, let $\|\cdot \|_0$, $\|\cdot \|_1$, $\|\cdot \|_\theta$ be norms on $\R^N$, and let $\|\cdot \|_0^*$, $\|\cdot \|_1^*$, $\|\cdot \|_\theta^*$ be norms in the corresponding dual spaces. We set $X_t=(\R^N, \, \|\cdot\|_t)$. Denote by $B_t$ be the unit ball in $X_t$, $t=0, \, 1, \, \theta$. Suppose that $$\|x\|^*_\theta \le (\|x\|_1^*)^{\theta}(\|x\|_0^*)^{1-\theta}, \quad x\in \R^N.$$ Then
$$
d_n(B_\theta, \, X_1) \le (d_n(B_0, \, X_1))^{1-\theta}.
$$
\end{trma}

\renewcommand{\proofname}{\bf Proof of Theorem \ref{fim_dim_poper}}

\begin{proof}
The equality \eqref{fin_dim_est2} can be checked directly. Let us prove \eqref{fin_dim_est1}.

{\it The upper estimate.} First we prove the estimate in the case when $p_j\le q_j$ for all $j=1, \, \dots, \, d$. Then $d_n(B^{k_1, \dots, k_d}_{p_1,\dots,p_d}, \, l^{k_1, \dots, k_d}_{q_1,\dots,q_d}) \le 1$.

Let $p_j\le 2$ for all $j=1, \, \dots, \, d$. Then $\mu=\nu=0$, and the right-hand side of \eqref{fin_dim_est2} is equal to $\min \{1, \, n^{-1/2}k_1^{1/q_1}\dots k_d^{1/q_d}\}$. We set $\hat q = \max _{1\le j\le d} q_j$. Applying Theorem \ref{glus}, we get 
\begin{align}
\label{est_sig0}
d_n(B^{k_1, \dots, k_d}_{p_1,\dots,p_d}, \, l^{k_1, \dots, k_d}_{q_1,\dots,q_d}) \le \prod _{j=1}^d k_j^{1/q_j-1/\hat q} d_n(B_2^{k_1\dots k_n}, \, l_{\hat q}^{k_1\dots k_n}) \underset{\overline{q}}{\lesssim} n^{-1/2}\prod _{j=1}^d k_j^{1/q_j}.
\end{align}
Hence in the case $p_j\le 2$ ($1\le j\le d$) the desired upper estimate is proved.

Let $p_j>2$ for some $j\in \{1, \, \dots, \, d\}$. Then, by \eqref{om_pq}, \eqref{upor} and \eqref{mu_nu_def}, we have $\omega_{p_{\sigma(1)}, q_{\sigma(1)}} < 1$. First we prove that 
\begin{align}
\label{est_sig1}
d_n(B^{k_1, \dots, k_d}_{p_1,\dots,p_d}, \, l^{k_1, \dots, k_d}_{q_1,\dots,q_d}) \underset{\overline{q}}{\lesssim} \Bigl(n^{-1/2}\prod _{j=1}^d k_j^{1/q_j} \Bigr)^{\omega_{p_{\sigma(1)}, q_{\sigma(1)}}}.
\end{align}
(Notice that if $\mu\ge 1$, we have $\omega_{p_{\sigma(1)}, q_{\sigma(1)}} = 0$ and the right-hand side of \eqref{est_sig1} is 1.)

Given $1\le p\le \infty$, we define the number $p'$ by the equation $\frac 1p + \frac{1}{p'} = 1$. By Theorem \ref{interp_ineq}, it suffices to prove that for all $x \in \R^{k_1\dots k_d}$
\begin{align}
\label{xp1_pd_2}
\|x\|_{l_{p_1', \dots,p_d'}^{k_1,\dots,k_d}} \le \|x\|_{l_{q_1', \dots,q_d'}^{k_1,\dots,k_d}}^{1-\omega_{p_{\sigma(1)}, q_{\sigma(1)}}} \|x\|_{l_{2, \dots,2}^{k_1,\dots,k_d}}^{\omega_{p_{\sigma(1)}, q_{\sigma(1)}}}.
\end{align}
Indeed, we define the numbers $\tilde p_j$ by the equation $\frac{1}{\tilde p_j} = \frac{1-\omega_{p_{\sigma(1)}, q_{\sigma(1)}}}{q_j} + \frac{\omega_{p_{\sigma(1)}, q_{\sigma(1)}}}{2}$, $1\le j\le d$. By \eqref{om_pq} and \eqref{upor}, we have $\tilde p'_j \le p'_j$; therefore,
$\|x\|_{l_{p_1', \dots,p_d'}^{k_1,\dots,k_d}} \le \|x\|_{l_{\tilde p_1', \dots,\tilde p_d'}^{k_1,\dots,k_d}}$. The relation $$\|x\|_{l_{\tilde p_1', \dots,\tilde p_d'}^{k_1,\dots,k_d}} \le \|x\|_{l_{q_1', \dots,q_d'}^{k_1,\dots,k_d}}^{1-\omega_{p_{\sigma(1)}, q_{\sigma(1)}}} \|x\|_{l_{2, \dots,2}^{k_1,\dots,k_d}}^{\omega_{p_{\sigma(1)}, q_{\sigma(1)}}}$$ follows from the H\"{o}lder's inequality (the proof is the same as for $d=2$; see Lemma 1 in \cite{vas_mix2}). Hence we get \eqref{xp1_pd_2}, which implies \eqref{est_sig1}.

Let now $\max\{2, \, \mu+1\}\le t\le d$. We show that
\begin{align}
\label{est_sig2} 
d_n(B^{k_1, \dots, k_d}_{p_1,\dots,p_d}, \, l^{k_1, \dots, k_d}_{q_1,\dots,q_d}) \underset{\overline{q}}{\lesssim} \prod _{j=1}^{t-1} k_{\sigma(j)} ^{1/q_{\sigma(j)} - 1/p^*_{\sigma(j)}} \Bigl(n^{-1/2}\prod _{j=1}^{t-1} k_{\sigma(j)}^{1/2} \prod _{j=t}^d k_{\sigma(j)}^{1/q_{\sigma(j)}}\Bigr) ^{\omega_{p_{\sigma(t)},q_{\sigma(t)}}}.
\end{align}

For $1\le j\le t-1$ we define the numbers $\hat p_{\sigma(j)}$ by
\begin{align}
\label{hat_p_def}
\frac{1}{\hat p_{\sigma(j)}} = \frac{1-\omega_{p_{\sigma(t)},q_{\sigma(t)}}}{q_{\sigma(j)}} + \frac{\omega_{p_{\sigma(t)},q_{\sigma(t)}}}{2}.
\end{align}
For $j\ge t$ we set $\hat p_{\sigma(j)} = p_{\sigma(j)}^*$. From \eqref{om_pq}, \eqref{upor} it follows that $\hat p_i \le p_i^*$, $1\le i\le d$. It remains to apply the inclusion $B^{k_1,\dots,k_d}_{p_1,\dots,p_d}\subset \prod_{j=1}^{t-1} k_{\sigma(j)}^{1/\hat p_{\sigma(j)} - 1/p^*_{\sigma(j)}} B^{k_1,\dots,k_d}_{\hat p_1,\dots,\hat p_d}$, the equalities $\omega_{\hat p_{\sigma(1)},q_{\sigma(1)}} = \dots = \omega_{\hat p_{\sigma(t-1)},q_{\sigma(t-1)}} =\omega_{p_{\sigma(t)},q_{\sigma(t)}}$, the estimates \eqref{est_sig0}, \eqref{est_sig1} for $d_n(B^{k_1, \dots, k_d}_{\hat{p}_1,\dots,\hat{p}_d}, \, l^{k_1, \dots, k_d}_{q_1,\dots,q_d})$ and the equality \eqref{hat_p_def}. This completes the proof of \eqref{est_sig2}.

Hence, 
\begin{align}
\label{up_low_tr} d_n(B^{k_1, \dots, k_d}_{p_1,\dots,p_d}, \, l^{k_1, \dots, k_d}_{q_1,\dots,q_d}) \underset{\overline{q}}{\asymp} \Phi_{\overline{p},\overline{q}}(k_1, \, \dots, \, k_d, \, n) \quad \text{if } \;p_j\le q_j, \; 1\le j\le d.
\end{align}

Now suppose that there is $j_0\in \{1, \, \dots, \, d\}$ such that $p_{j_0} > q_{j_0}$. Let $p^{**}_i = \min \{p_i, \, q_i\}$, $1\le i\le d$. We apply the inclusion $$B^{k_1,\dots,k_d}_{p_1,\dots,p_d} \subset \prod_{j=1}^{\mu} k_{\sigma(j)}^{1/q_{\sigma(j)}-1/p_{\sigma(j)}}B^{k_1,\dots,k_d}_{p^{**}_1,\dots,p^{**}_d}$$ and the estimate \eqref{up_low_tr} for $B^{k_1,\dots,k_d}_{p^{**}_1,\dots,p^{**}_d}$. In addition, we notice that, for $\nu=d$,
$$\Phi_{\overline{p},\overline{q}}(k_1, \, \dots, \, k_d, \, n) =
 \prod _{p_j>q_j} k_j^{1/q_j - 1/p_j} \times$$$$ \times\min \Bigl\{ 1, \, \min _{\omega_{p_t,q_t}\in (0, \, 1)} \prod _{0<\omega_{p_j,q_j}<\omega_{p_t,q_t}} k_j^{1/q_j-1/p_j}(n^{-1/2} \prod _{\omega_{p_j,q_j}< \omega_{p_t,q_t}} k_j^{1/2}\prod _{\omega_{p_j,q_j}\ge \omega_{p_t,q_t}} k_j^{1/q_j})^{\omega _{p_t,q_t}}\Bigl\},
$$
and for $\nu<d$,
$$\Phi_{\overline{p},\overline{q}}(k_1, \, \dots, \, k_d, \, n) =
 \prod _{p_j>q_j} k_j^{1/q_j - 1/p_j} \times$$$$ \times\min \Bigl\{ 1, \, \min _{\omega_{p_t,q_t}\in (0, \, 1)} \prod _{0<\omega_{p_j,q_j}<\omega_{p_t,q_t}} k_j^{1/q_j-1/p_j}(n^{-1/2} \prod _{\omega_{p_j,q_j}< \omega_{p_t,q_t}} k_j^{1/2}\prod _{\omega_{p_j,q_j}\ge \omega_{p_t,q_t}} k_j^{1/q_j})^{\omega _{p_t,q_t}}, $$$$\prod _{\omega_{p_j,q_j}\in (0, \, 1)}k_j^{1/q_j - 1/p_j} n^{-1/2} \prod_{\omega_{p_j,q_j}<1} k_j^{1/2}\prod_{\omega_{p_j,q_j}=1}k_j^{1/q_j}\Bigl\}.
$$
We write the similar formula for $\Phi_{\overline{p}^{**},\overline{q}}(k_1, \, \dots, \, k_d, \, n)$, where $\overline{p}^{**} = (p_1^{**}, \, \dots, \, p_d^{**})$. Notice that $\omega_{p_j,q_j}\ne \omega_{p_j^{**},q_j}$ only if $2=q_j<p_j$; in this case, $\omega_{p_j,q_j}=0$, $\omega_{p_j^{**},q_j}=1$. It remains to compare the formulas for $\Phi_{\overline{p},\overline{q}}(k_1, \, \dots, \, k_d, \, n)$
and for $$\prod_{p_j>q_j} k_{j}^{1/q_j-1/p_j}\Phi_{\overline{p}^{**},\overline{q}}(k_1, \, \dots, \, k_d, \, n),$$ and notice that $$\prod_{q_j=2< p_j}k_j^{1/2}\prod_{q_j=2= p_j}k_j^{1/q_j}=\prod _{q_j=2= p_j^{**}}k_j^{1/q_j}.$$

\vskip 0.3cm

{\it The lower estimate.} Let $$G = \{(\sigma_1, \, \dots, \, \sigma_d, \, \varepsilon_1, \, \dots, \, \varepsilon_d): \sigma _j \in S_{k_j}, \, \varepsilon_j \in \{-1, \, 1\}^{k_j}, \, 1\le j\le d\},$$
where $S_{k_j}$ is the group of permutations of $k_j$ elements.
Given $g = (\sigma_1, \, \dots, \, \sigma_d, \, \varepsilon_1, \, \dots, \, \varepsilon_d) \in G$, $\varepsilon_j = (\varepsilon_{j,1}, \, \dots, \, \varepsilon_{j,k_j})$, $x = (x_{i_1,\dots,i_d})_{1\le i_1\le k_1,\dots, 1\le i_d\le k_d}$, we set $$g(x) = (\varepsilon_{1,i_1}\dots \varepsilon _{d,i_d}x_{\sigma_1(i_1),\dots,\sigma_d(i_d)})_{1\le i_1\le k_1,\dots, 1\le i_d\le k_d}.$$

Let $s_j \in \{1, \, \dots, \, k_j\}$, $1\le j\le d$, $\overline{s}=(s_1, \, \dots, \, s_d)$. We set $\hat x(\overline{s}) = (\hat x_{i_1,\dots,i_d})_{1\le i_1\le k_1,\dots, 1\le i_d\le k_d}$, where
$$
\hat x_{i_1,\dots,i_d} = \begin{cases} 1 & \text{if }1\le i_1\le s_1, \, \dots, \, 1\le i_d\le s_d, \\ 0 & \text{otherwise},\end{cases}
$$
$$
V^{k_1,\dots,k_d}_{s_1,\dots,s_d} = {\rm conv}\, \{g(\hat x(\overline{s})):\; g\in G\}.
$$

\begin{Lem}
\label{a_q_low}
Let $2\le q_j<\infty$, $k_j\in \N$, $s_j\in \{1, \, \dots, \, k_j\}$, $1\le j\le d$. Then there exist numbers $a(\overline{q})>0$, $b(\overline{q})>0$ such that, for $n\le a(\overline{q}) k_1^{\frac{2}{q_1}}\dots k_d^{\frac{2}{q_d}} s_1^{1-\frac{2}{q_1}}\dots s_d^{1-\frac{2}{q_d}}$, 
$$
d_n(V^{k_1,\dots,k_d}_{s_1,\dots,s_d}, \, l_{q_1, \, \dots, \, q_d}^{k_1,\dots,k_d}) \ge b(\overline{q}) s_1^{1/q_1} \dots s_d^{1/q_d}.
$$
The function $a(\cdot)$ is non-increasing in each $q_j$, $1\le j\le d$, the function $b(\cdot)$ is continuous.
\end{Lem}

\renewcommand{\proofname}{\bf Proof}

\begin{proof}
We generalize the arguments in \cite{gluskin1, vas_besov}. For brevity, we write $\hat x = \hat x(\overline{s})$. First we prove that there is $c(\overline{q})>0$ such that for each $x=(x_{i_1,\dots,i_{d}})_{1\le i_1\le k_1,\dots,1\le i_{d}\le k_{d}}$
\begin{align}
\label{f_x_q} \begin{array}{c}\|g(\hat x)-x\|^{q_d}_{l_{q_1,\dots,q_{d}}^{k_1,\dots,k_{d}}} \ge \frac{s_1^{q_d/q_1}\dots s_{d-1}^{q_d/q_{d-1}}s_d}{2} - \\ -q_ds_1^{\frac{q_d}{q_1}-1} \dots s_{d-1}^{\frac{q_d}{q_{d-1}}-1} \sum \limits _{1\le i_1\le k_1, \dots, 1\le i_{d}\le k_{d}} g(\hat x)_{i_1,\dots,i_d}x_{i_1,\dots,i_d} + c(\overline{q})\|x\|^{q_d}_{l_{q_1,\dots,q_{d}}^{k_1,\dots,k_{d}}},\end{array}
\end{align}
and the function $c(\cdot)$ is continuous.

Indeed, let $\varphi(x) = \|g(\hat x)-x\|^{q_d}_{l_{q_1,\dots,q_{d}}^{k_1,\dots,k_{d}}}$. Then the function $\varphi$ is convex, $\varphi(0) = s_1^{q_d/q_1}\dots s_{d-1}^{q_d/q_{d-1}}s_d$, $\varphi'_{x_{i_1,\dots,i_d}}(0) = - q_ds_1^{\frac{q_d}{q_1}-1} \dots s_{d-1}^{\frac{q_d}{q_{d-1}}-1}g(\hat x)_{i_1,\dots,i_d}$. Hence
$$
\|g(\hat x)-x\|^{q_d}_{l_{q_1,\dots,q_{d}}^{k_1,\dots,k_{d}}} \ge s_1^{q_d/q_1}\dots s_{d-1}^{q_d/q_{d-1}}s_d - $$$$ - q_ds_1^{\frac{q_d}{q_1}-1} \dots s_{d-1}^{\frac{q_d}{q_{d-1}}-1} \sum \limits _{1\le i_1\le k_1, \dots, 1\le i_{d}\le k_{d}} g(\hat x)_{i_1,\dots,i_d}x_{i_1,\dots,i_d}.
$$
For each $t>0$ we set $c_0(\overline{q}, \, t)=\frac{1}{2t^{q_d}}$. Then for all $x$ such that $\|x\|_{l_{q_1,\dots,q_{d}}^{k_1,\dots,k_{d}}} \le t s_1^{1/q_1}\dots s_{d}^{1/q_{d}}$ we have
$$
\|g(\hat x)-x\|^{q_d}_{l_{q_1,\dots,q_{d}}^{k_1,\dots,k_{d}}} \ge \frac{s_1^{q_d/q_1}\dots s_{d-1}^{q_d/q_{d-1}}s_d}{2} - $$$$-q_ds_1^{\frac{q_d}{q_1}-1} \dots s_{d-1}^{\frac{q_d}{q_{d-1}}-1} \sum \limits _{1\le i_1\le k_1, \dots, 1\le i_{d}\le k_{d}} g(\hat x)_{i_1,\dots,i_d}x_{i_1,\dots,i_d} + c_0(\overline{q}, \, t)\|x\|^{q_d}_{l_{q_1,\dots,q_{d}}^{k_1,\dots,k_{d}}}.
$$
Further, there is $t_0(\overline{q})>0$ continuously depending on $\overline{q}$ such that if $\|x\|_{l_{q_1,\dots,q_{d}}^{k_1,\dots,k_{d}}} \ge t_0(\overline{q}) s_1^{1/q_1}\dots s_d^{1/q_d}$, then
\begin{align}
\label{gxxqdge}
\begin{array}{c}
\|g(\hat x)-x\|^{q_d}_{l_{q_1,\dots,q_{d}}^{k_1,\dots,k_{d}}} \ge \frac{s_1^{q_d/q_1}\dots s_{d-1}^{q_d/q_{d-1}}s_d}{2} -
\\
-q_ds_1^{\frac{q_d}{q_1}-1} \dots s_{d-1}^{\frac{q_d}{q_{d-1}}-1} \sum \limits _{1\le i_1\le k_1, \dots, 1\le i_{d}\le k_{d}} g(\hat x)_{i_1,\dots,i_d}x_{i_1,\dots,i_d} + 2^{-q_d}\|x\|^{q_d}_{l_{q_1,\dots,q_{d}}^{k_1,\dots,k_{d}}};
\end{array}
\end{align}
this implies \eqref{f_x_q}.

Let us prove \eqref{gxxqdge}. Indeed, from the inequality $(a+b)^{q_d}\le 2^{q_d-1}(a^{q_d} + b^{q_d})$ it follows that
$\|g(\hat x)-x\|^{q_d}_{l_{q_1,\dots,q_{d}}^{k_1,\dots,k_{d}}}+ s_1^{q_d/q_1}\dots s_{d-1}^{q_d/q_{d-1}} s_d \ge 2^{-q_d+1} \|x\|^{q_d}_{l_{q_1,\dots,q_{d}}^{k_1,\dots,k_{d}}}$; hence it suffices to prove that
$$
2^{-q_d} \|x\|_{l_{q_1,\dots,q_{d}}^{k_1,\dots,k_{d}}}^{q_d} \ge 2 s_1^{q_d/q_1}\dots s_{d-1}^{q_d/q_{d-1}}s_d - $$$$-q_ds_1^{\frac{q_d}{q_1}-1} \dots s_{d-1}^{\frac{q_d}{q_{d-1}}-1} \sum \limits _{1\le i_1\le k_1, \dots, 1\le i_{d}\le k_{d}} g(\hat x)_{i_1,\dots,i_d}x_{i_1,\dots,i_d}.
$$
By the H\"{o}lder's inequality, it suffices to check that
\begin{align}
\label{2qdnxqd}
2^{-q_d} \|x\|_{l_{q_1,\dots,q_{d}}^{k_1,\dots,k_{d}}}^{q_d} \ge 2 s_1^{q_d/q_1}\dots s_{d-1}^{q_d/q_{d-1}}s_d + q_ds_1^{\frac{q_d-1}{q_1}} \dots s_{d-1}^{\frac{q_d-1}{q_{d-1}}} s_d^{\frac{q_d-1}{q_d}}\|x\|_{l_{q_1,\dots,q_{d}}^{k_1,\dots,k_{d}}}.
\end{align}
We define the number $t\ge 0$ by the equation $\|x\|_{l_{q_1,\dots,q_{d}}^{k_1,\dots,k_{d}}}= t s_1^{1/q_1}\dots s_{d}^{1/q_{d}}$. Then \eqref{2qdnxqd} is equivalent to $2^{-q_d} t^{q_d} \ge 2 +q_dt$; this inequality holds for sufficiently large $t$ (the smallest value of such positive numbers $t$ is continuous in $\overline{q}$). This completes the proof of \eqref{f_x_q}.

Let $L\subset l_{q_1,\dots,q_d}^{k_1, \, \dots, \, k_d}$ be a subspace of dimension $n$. We denote by $y_g$ the nearest to $g(\hat x)$ element of $L$ with respect to the norm $\|\cdot\|_{l_{q_1,\dots,q_{d}}^{k_1,\dots,k_{d}}}$. Averaging the both sides of \eqref{f_x_q} over $g\in G$, we get
\begin{align}
\label{l_e_max_g}
\begin{array}{c}
\max _{g\in G} \|g(\hat x) - y_g\|_{l_{q_1,\dots,q_d}^{k_1, \, \dots, \, k_d}} ^{q_d} \ge \frac{s_1^{q_d/q_1}\dots s_{d-1}^{q_d/q_{d-1}}s_d}{2} + \frac{c(\overline{q})}{|G|} \sum \limits _{g\in G}\| y_g\|_{l_{q_1,\dots,q_d}^{k_1, \, \dots, \, k_d}} ^{q_d} - \\-q_d\frac{s_1^{\frac{q_d}{q_1}-1}\dots s_{d-1}^{\frac{q_d}{q_{d-1}}-1}}{|G|} \sum \limits _{g\in G} \sum \limits_{1\le i_1\le k_1, \dots, \, 1\le i_d\le k_d} g(\hat x)_{i_1,\dots,i_d} (y_g)_{i_1,\dots,i_d}.
\end{array}
\end{align}
Similarly as in \cite{gluskin1}, \cite[pp. 15--16]{vas_besov} we obtain the inequality
\begin{align}
\label{sum_g_ineq} \begin{array}{c}\frac{1}{|G|}\left|\sum \limits _{g\in G} \sum \limits_{1\le i_1\le k_1, \dots, \, 1\le i_d\le k_d} g(\hat x)_{i_1,\dots,i_d} (y_g)_{i_1,\dots,i_d}\right| \le \\ \le \frac{n^{1/2}s_1^{1/2}\dots s_d^{1/2}}{k_1^{1/2} \dots k_d^{1/2}} \left(\frac{1}{|G|}\sum \limits _{g\in G} \sum \limits_{1\le i_1\le k_1, \dots, \, 1\le i_d\le k_d}|(y_g)_{i_1,\dots,i_d}|^2\right)^{1/2}.\end{array}
\end{align}
Let $n^{1/2}\le \tilde a k_1^{1/q_1}\dots k_d^{1/q_d}s_1^{1/2-1/q_1} \dots s_d^{1/2-1/q_d}$, where $\tilde a>0$.
Since $q_j\ge 2$, $j=1, \, \dots, \, d$, the right-hand side of \eqref{sum_g_ineq} is not greater than
$$
\frac{n^{1/2}s_1^{1/2}\dots s_d^{1/2}}{k_1^{1/q_1} \dots k_d^{1/q_d}} \left(\frac{1}{|G|}\sum \limits _{g\in G} \|y_g\|_{l_{q_1,\dots,q_d}^{k_1,\dots,k_d}}^{q_d}\right)^{1/q_d}\le $$$$
\le \tilde a s_1^{1-1/q_1}\dots s_d^{1-1/q_d}\left(\frac{1}{|G|}\sum \limits _{g\in G} \|y_g\|_{l_{q_1,\dots,q_d}^{k_1,\dots,k_d}}^{q_d}\right)^{1/q_d}.
$$
Hence
$$
\left|\frac{s_1^{\frac{q_d}{q_1}-1}\dots s_{d-1}^{\frac{q_d}{q_{d-1}}-1}}{|G|} \sum \limits _{g\in G} \sum \limits_{1\le i_1\le k_1, \dots, \, 1\le i_d\le k_d} g(\hat x)_{i_1,\dots,i_d} (y_g)_{i_1,\dots,i_d}\right|\le
$$
$$
\le \tilde a s_1^{\frac{q_d-1}{q_1}}\dots s_d^{\frac{q_d-1}{q_d}}\left(\frac{1}{|G|}\sum \limits _{g\in G} \|y_g\|_{l_{q_1,\dots,q_d}^{k_1,\dots,k_d}}^{q_d}\right)^{1/q_d}.
$$
This together with \eqref{l_e_max_g} yields
\begin{align}
\label{max_g_young}
\begin{array}{c}
\max _{g\in G} \|g(\hat x) - y_g\|_{l_{q_1,\dots,q_d}^{k_1, \, \dots, \, k_d}} ^{q_d} \ge \frac{s_1^{q_d/q_1}\dots s_{d-1}^{q_d/q_{d-1}}s_d}{2} + \frac{c(\overline{q})}{|G|} \sum \limits _{g\in G}\| y_g\|_{l_{q_1,\dots,q_d}^{k_1, \, \dots, \, k_d}} ^{q_d} - 
\\
-q_d\tilde a s_1^{\frac{q_d-1}{q_1}}\dots s_d^{\frac{q_d-1}{q_d}}\left(\frac{1}{|G|}\sum \limits _{g\in G} \|y_g\|_{l_{q_1,\dots,q_d}^{k_1,\dots,k_d}}^{q_d}\right)^{1/q_d}.
\end{array}
\end{align}
It remains to apply the Young's inequality $\xi \eta \le \frac{|\xi|^{q_d}}{q_d} + \frac{|\eta|^{q_d'}}{q_d'}$ for 
$$
\xi = \tilde a^{1/2}\left(\frac{1}{|G|}\sum \limits _{g\in G} \|y_g\|_{l_{q_1,\dots,q_d}^{k_1,\dots,k_d}}^{q_d}\right)^{1/q_d}, \quad \eta = \tilde a^{1/2}s_1^{\frac{q_d-1}{q_1}}\dots s_d^{\frac{q_d-1}{q_d}}
$$
and to choose a sufficiently small $\tilde a = \tilde a(\overline{q})$, continuously depending on $\overline{q}$, such that the right-hand side of \eqref{max_g_young} is estimated from below by $s_1^{q_d/q_1}\dots s_{d-1}^{q_d/q_{d-1}}s_d/4$. After that we set $a(\overline{q}) = \min \{\tilde a(\tilde q_1, \, \dots, \, \tilde q_d):\; 2\le \tilde q_1\le q_1, \, \dots, \, 2\le \tilde q_d\le q_d\}$. 
\end{proof}

\begin{Rem}
\label{qj2} If $q_1=\dots=q_d=2$, we can use the formula for the square of a sum instead of \eqref{f_x_q}; after that we apply \eqref{sum_g_ineq} and get, for some $t\ge 0$,
$$
\max _{g\in G} \|g(\hat x) - y_g\|_{l_{2,\dots,2}^{k_1, \, \dots, \, k_d}} ^2 \ge s_1\dots s_d-2 \frac{n^{1/2}s_1^{1/2}\dots s_d^{1/2}}{k_1^{1/2}\dots k_d^{1/2}} t +t^2 \ge s_1\dots s_d\left(1-\frac{n}{k_1\dots k_d}\right);
$$
hence
$$
d_n(V_{s_1, \, \dots, \, s_d}^{k_1, \, \dots, \, k_d}, \, l_2^{k_1\dots k_d}) \ge s_1^{1/2}\dots s_d^{1/2} \sqrt{1-\frac{n}{k_1\dots k_d}}.
$$
\end{Rem}

Applying Lemma \ref{a_q_low} and Remark \ref{qj2} and generalizing arguments from \cite[formula (9)]{vas_mix2} to the $d$-dimensional case, we get
\begin{Cor}
\label{v_dn_n} Let $2\le q_j<\infty$, $k_j\in \N$, $s_j\in \{1, \, \dots, \, k_j\}$, $1\le j\le d$, $n\in \Z_+$, $n\le \frac{k_1\dots k_d}{2}$. Then 
\begin{align} \label{vk1kddn}
d_n(V^{k_1,\dots,k_d}_{s_1,\dots,s_d}, \, l_{q_1, \, \dots, \, q_d}^{k_1,\dots,k_d}) \underset{\overline{q}}{\gtrsim} \begin{cases} s_1^{1/q_1} \dots s_d^{1/q_d}, & n \le k_1^{\frac{2}{q_1}}\dots k_d^{\frac{2}{q_d}} s_1^{1-\frac{2}{q_1}}\dots s_d^{1-\frac{2}{q_d}}, \\ n^{-1/2}k_1^{1/q_1}\dots k_d^{1/q_d}s_1^{1/2}\dots s_d^{1/2}, & n \ge k_1^{\frac{2}{q_1}}\dots k_d^{\frac{2}{q_d}} s_1^{1-\frac{2}{q_1}}\dots s_d^{1-\frac{2}{q_d}}. \end{cases}
\end{align}
\end{Cor}

Now we obtain the lower estimates of $d_n(B^{k_1, \dots, k_d}_{p_1,\dots,p_d}, \, l^{k_1, \dots, k_d}_{q_1,\dots,q_d})$ for $n\le \frac{k_1\dots k_d}{2}$.
We use the inclusion
\begin{align}
\label{v_in_b}
s_1^{-1/p_1}\dots s_d^{-1/p_d} V^{k_1,\dots,k_d}_{s_1,\dots,s_d} \subset B^{k_1, \dots, k_d}_{p_1,\dots,p_d}
\end{align}
and \eqref{vk1kddn}. The numbers $s_1, \, \dots, \, s_d$ are defined as follows.

Let $n\le k_{\sigma(1)} \dots k_{\sigma(\mu)} k_{\sigma(\mu+1)}^{2/q_{\sigma(\mu+1)}} \dots k_{\sigma(d)}^{2/q_{\sigma(d)}}$. We set $s_{\sigma(j)} = k_{\sigma(j)}$ for $1\le j\le \mu$, $s_{\sigma(j)} = 1$ for $\mu+1\le j\le d$. Then from \eqref{vk1kddn} and \eqref{v_in_b} it follows that
\begin{align}
\label{dn11}
d_n(B^{k_1, \dots, k_d}_{p_1,\dots,p_d}, \, l^{k_1, \dots, k_d}_{q_1,\dots,q_d}) \underset{\overline{q}}{\gtrsim} k_{\sigma(1)}^{1/q_{\sigma(1)} - 1/p_{\sigma(1)}} \dots k_{\sigma(\mu)}^{1/q_{\sigma(\mu)} - 1/p_{\sigma(\mu)}}.
\end{align}

Let $t\in \{\mu+1, \, \mu+2, \, \dots, \, \nu\}$,
\begin{align}
\label{ks12qt}
k_{\sigma(1)} \dots k_{\sigma(t-1)} k_{\sigma(t)}^{2/q_{\sigma(t)}} \dots k_{\sigma(d)}^{2/q_{\sigma(d)}} < n \le k_{\sigma(1)} \dots k_{\sigma(t)} k_{\sigma(t+1)}^{2/q_{\sigma(t+1)}} \dots k_{\sigma(d)}^{2/q_{\sigma(d)}}.
\end{align}
Then $q_{\sigma(t)}>2$; by \eqref{om_pq}, \eqref{upor} and \eqref{mu_nu_def}, we have $p_{\sigma(j)}>2$, $j=1, \, \dots, \, t$.
Let $$l = (n^{1/2}k_{\sigma(1)}^{-1/2} \dots k_{\sigma(t-1)}^{-1/2} k_{\sigma(t)}^{-1/q_{\sigma(t)}}\dots k_{\sigma(d)}^{-1/q_{\sigma(d)}})^{\frac{1}{1/2-1/q_{\sigma(t)}}}.$$
By \eqref{ks12qt}, we have $1\le l\le k_{\sigma(t)}$. We set $s_{\sigma(j)}=k_{\sigma(j)}$ for $1\le j\le t-1$, $s_{\sigma(t)} = \lceil l \rceil$, $s_{\sigma(j)} = 1$ for $t+1\le j\le d$. Then $1\le s_i\le k_i$, $1\le i\le d$, $$n \le k_{\sigma(1)}^{2/q_{\sigma(1)}}\dots k_{\sigma(d)}^{2/q_{\sigma(d)}} s_{\sigma(1)}^{1-2/q_{\sigma(1)}} \dots s_{\sigma(d)}^{1-2/q_{\sigma(d)}}.$$
This together with \eqref{vk1kddn} and \eqref{v_in_b} implies that
\begin{align}
\label{dn22}
\begin{array}{c}
d_n(B^{k_1, \dots, k_d}_{p_1,\dots,p_d}, \, l^{k_1, \dots, k_d}_{q_1,\dots,q_d}) \underset{\overline{q}}{\gtrsim} k_{\sigma(1)}^{1/q_{\sigma(1)} - 1/p_{\sigma(1)}} \dots k_{\sigma(t-1)}^{1/q_{\sigma(t-1)} - 1/p_{\sigma(t-1)}} \times \\ \times
(n^{-1/2}k_{\sigma(1)}^{1/2} \dots k_{\sigma(t-1)}^{1/2} k_{\sigma(t)}^{1/q_{\sigma(t)}}\dots k_{\sigma(d)}^{1/q_{\sigma(d)}})^{\frac{1/p_{\sigma(t)}-1/q_{\sigma(t)}}{1/2-1/q_{\sigma(t)}}}.
\end{array}
\end{align}

For $\nu=d$ the desired lower estimate in \eqref{fin_dim_est1} follows from \eqref{dn11}, \eqref{dn22}.

Let $\nu<d$, $n> k_{\sigma(1)}\dots k_{\sigma(\nu)} k^{2/q_{\sigma(\nu+1)}}_{\sigma(\nu+1)} \dots k^{2/q_{\sigma(d)}} _{\sigma(d)}$. We set $s_{\sigma(j)} = k_{\sigma(j)}$ for $j\le \nu$, $s_{\sigma(j)} = 1$ for $j\ge \nu+1$. Then $n \ge k_{\sigma(1)}^{2/q_{\sigma(1)}}\dots k_{\sigma(d)}^{2/q_{\sigma(d)}} s_{\sigma(1)}^{1-2/q_{\sigma(1)}} \dots s_{\sigma(d)}^{1-2/q_{\sigma(d)}}$. From \eqref{vk1kddn} and \eqref{v_in_b} we get that
\begin{align}
\label{dn33}
\begin{array}{c}
d_n(B^{k_1, \dots, k_d}_{p_1,\dots,p_d}, \, l^{k_1, \dots, k_d}_{q_1,\dots,q_d}) \underset{\overline{q}}{\gtrsim} k_{\sigma(1)}^{1/q_{\sigma(1)} - 1/p_{\sigma(1)}} \dots k_{\sigma(\nu)}^{1/q_{\sigma(\nu)} - 1/p_{\sigma(\nu)}} \times
\\
\times k_{\sigma(\nu+1)}^{1/q_{\sigma(\nu+1)}}\dots k_{\sigma(d)}^{1/q_{\sigma(d)}}n^{-1/2} k_{\sigma(1)}^{1/2} \dots k_{\sigma(\nu)}^{1/2}.
\end{array}
\end{align}
Now the lower estimate follows from \eqref{dn11}, \eqref{dn22}, \eqref{dn33} and \eqref{fin_dim_est2}.
This completes the proof.
\end{proof}

\renewcommand{\proofname}{\bf Proof of Proposition \ref{fim_dim_poper2}}

The following theorem was proved in \cite{mal_rjut}.
\begin{trma}
\label{mr_teor} {\rm (see \cite{mal_rjut}).} Let $m$, $k\in \N$, $n\le \frac{mk}{2}$, $V=(B_1^m)^k$ (the Cartesian product of $k$ copies of $B_1^m$). Then
$$
d_n(V, \, l^{m,k}_{2,1}) \gtrsim k.
$$
\end{trma}

\begin{proof}
The upper estimate follows from the inclusion $B^{k_1,\dots,k_d}_{p_1, \, \dots, \, p_d} \subset k_{\nu+1}^{1/q_{\nu+1}-1/p_{\nu+1}} \dots k_d^{1/q_d-1/p_d}B^{k_1,\dots,k_d}_{q_1, \, \dots, \, q_d}$.

Let us prove the lower estimate. Applying the inclusion $$k_{\nu+1}^{-1/p_{\nu+1}} \dots k_d^{-1/p_d}(B_1^{k_1\dots k_\nu})^{k_{\nu+1}\dots k_d} \subset B^{k_1,\dots,k_d}_{p_1, \, \dots, \, p_d},$$ taking into account the conditions $q_j\le 2$ for $j\le \nu$ and applying Theorem \ref{mr_teor}, we get
$$
d_n(B^{k_1,\dots,k_d}_{p_1, \, \dots, \, p_d}, \, l^{k_1,\dots,k_d}_{q_1, \, \dots, \, q_d}) \ge $$$$ \ge k_{\nu+1}^{-1+1/q_{\nu+1}-1/p_{\nu+1}} \dots k_d^{-1+1/q_d-1/p_d} d_n((B_1^{k_1\dots k_\nu})^{k_{\nu+1}\dots k_d}, \, l^{k_1\dots k_\nu, k_{\nu+1}\dots k_d}_{2,1}) \gtrsim$$$$\gtrsim k_{\nu+1}^{1/q_{\nu+1}-1/p_{\nu+1}} \dots k_d^{1/q_d-1/p_d}.
$$
This completes the proof.
\end{proof}

\renewcommand{\proofname}{\bf Proof}

\section{The discretization}

We follow the book \cite{teml_book}. For $m\in \N$, we denote by ${\cal K}_{m-1}(x) = \frac{\sin^2(mx/2)}{m \sin^2(x/2)}$ Fej\'er kernel, and by ${\cal V}_m =2{\cal K}_{2m-1}-{\cal K}_{m-1}$, de la Vall\'ee Poussin kernel. For $\overline{N} = (N_1, \, \dots, \, N_d)\in \Z_+^d$, the multivariate Fej\'er and de la Vall\'ee Poussin kernels are defined, respectively, by formulas $${\cal K}_{\overline{N}}(x_1, \, \dots, \, x_d) = \prod _{j=1}^d {\cal K}_{N_j}(x_j), \quad {\cal V}_{\overline{N}}(x_1, \, \dots, \, x_d) = \prod _{j=1}^d {\cal V}_{N_j}(x_j),$$
where $(x_1, \, \dots, \, x_d)\in \mathbb{R}^d$.
By $V_{\overline{N}}$ we denote the convolution operator with the kernel ${\cal V}_{\overline{N}}$; i.e., 
\begin{align}
\label{vn_def_sv}
V_{\overline{N}}f(x)= \frac{1}{(2\pi)^d}\int \limits _{\mathbb{T}^d} {\cal V}_{\overline{N}}(x-y)f(y)\, dy.
\end{align}

We write $\overline{1}=(1, \, 1, \, \dots, \, 1)\in \N^d$. 

By ${\cal T}(\overline{N}, \, d)$ we denote the space of trigonometric polynomials with harmonics in $\prod _{j=1}^d [-N_j, \, N_j]$. Then
\begin{align}
\label{vp_im} {\rm Im}\, V_{\overline{N}} \subset {\cal T}(2\overline{N}-\overline{1}, \, d).
\end{align}

Let $\overline{r}=(r_1, \, \dots, \, r_d)$, $r_j>0$, $1\le j\le d$. We set
\begin{align}
\label{beta_vec_def}
\overline{\beta} = (\beta_1, \, \dots, \, \beta_d), \; \text{where }\; \beta_j = \frac{1/r_j}{\sum \limits _{i=1}^d 1/r_i} \stackrel{\eqref{sredn_1}, \eqref{sredn_2}}{=} \frac{\langle \overline{r} \rangle}{d}\cdot \frac{1}{r_j}.
\end{align}
For $m\in \Z_+$, we set 
\begin{align}
\label{2beta_vec_def}
\lfloor 2^{\overline{\beta}m}\rfloor = (\lfloor2^{\beta_1m}\rfloor, \, \dots, \, \lfloor2^{\beta_dm}\rfloor).
\end{align}

\subsection{The upper estimate}

Let $m\in \Z_+$, let the vectors $\overline{\beta}$ and $\lfloor 2^{\overline{\beta}m}\rfloor$ be defined by formulas \eqref{beta_vec_def}, \eqref{2beta_vec_def}. We set $V(\overline{r}, \, m) = V_{\lfloor 2^{\overline{\beta}m}\rfloor}$ (see \eqref{vn_def_sv}), 
$$
A(\overline{r},m) = \begin{cases} V(\overline{r}, \, m)-V(\overline{r}, \, m-1) & \text{for }m\in \N, \\ V(\overline{r}, \, 0) & \text{for }m=0.\end{cases}
$$

\begin{trma}
\label{appr_teor}
{\rm (see \cite{teml_book}, Corollary 3.4.8).} Let $\overline{r}=(r_1, \, \dots, \, r_d)$, $\overline{p}=(p_1, \, \dots, \, p_d)$, $r_j>0$, $1\le p_j\le \infty$, $1\le j\le d$. Then for all functions $f\in H^{\overline{r}}_{\overline{p}}(\mathbb{T}^d)$ the following estimates hold:
\begin{align}
\label{fvfrm2rdm}
\|f-V(\overline{r},m)f\|_{L_{\overline{p}}(\mathbb{T}^d)} \underset{\overline{r},d}{\lesssim} 2^{-\frac{\langle \overline{r}\rangle}{d}m},\quad
\|A(\overline{r},m)f\|_{L_{\overline{p}}(\mathbb{T}^d)} \underset{\overline{r},d}{\lesssim} 2^{-\frac{\langle \overline{r}\rangle}{d}m}.
\end{align}
\end{trma}

Given $a\in \R$, we write $a_+=\max\{a, \, 0\}$.

\begin{trma}
\label{nik_ineq}
{\rm (the Nikol'skii inequality; see \cite{teml_book}, Theorem 3.3.2; \cite{unin}).} Let $\overline{N}=(N_1, \, \dots, \, N_d)$, $\overline{p}=(p_1, \, \dots, \, p_d)$, $\overline{q}=(q_1, \, \dots, \, q_d)$, $N_j\in \N$, $1\le p_j, \, q_j\le \infty$, $1\le j\le d$. Then, for all trigonometric polynomials $t\in {\cal T}(\overline{N}, \, d)$,
$$
\|t\|_{L_{\overline{q}}(\mathbb{T}^d)} \underset{d}{\lesssim} \|t\|_{L_{\overline{p}}(\mathbb{T}^d)} \prod _{j=1}^d N_j^{(1/p_j-1/q_j)_+}.
$$
\end{trma}

We set
\begin{align}
\label{gamma_0_gamma}
\gamma_0(\overline{p}, \, \overline{q}, \, \overline{r}) = 1- \sum \limits _{j=1}^d \frac{1}{r_j}\Bigl( \frac{1}{p_j}-\frac{1}{q_j}\Bigr)_+, \quad \gamma(\overline{p}, \, \overline{q}, \, \overline{r}) = 1- \sum \limits _{j=1}^d \frac{1}{r_j}\Bigl( \frac{1}{p_j}-\frac{1}{q_j}\Bigr).
\end{align}

From \eqref{vp_im}, \eqref{beta_vec_def}, \eqref{2beta_vec_def} and Theorems \ref{appr_teor}, \ref{nik_ineq} it follows that, for all functions $f\in H^{\overline{r}}_{\overline{p}}(\mathbb{T}^d)$,
\begin{align}
\label{arm_g0}
\|A(\overline{r}, \, m)f\| _{L_{\overline{q}}(\mathbb{T}^d)} \underset{\overline{r},d}{\lesssim} 2^{-m\gamma_0(\overline{p}, \, \overline{q}, \, \overline{r})\langle \overline{r}\rangle/d}. 
\end{align}
By the assumptions of Theorems \ref{main1} and \ref{main2} (see \eqref{emb_cond}, \eqref{theta_posit}), we have $\gamma_0(\overline{p}, \, \overline{q}, \, \overline{r})>0$; hence by \eqref{fvfrm2rdm} and \eqref{arm_g0} for all functions $f\in H^{\overline{r}}_{\overline{p}}(\mathbb{T}^d)$ we have
\begin{align}
\label{series} f = \sum \limits _{m\in \Z_+} A(\overline{r}, \, m)f,
\end{align}
where the series converges in $L_{\overline{q}}(\mathbb{T}^d)$.

We set $$M_{m,\overline{p},\overline{r}} = \{A(\overline{r}, \, m)f:\; f\in H^{\overline{r}}_{\overline{p}}(\mathbb{T}^d)\}.$$

Let $\nu_m\in \Z_+$, $m\in \Z_+$, $\sum \limits _{m\in \Z_+} \nu_m = k$. From \eqref{series} it follows that
\begin{align}
\label{dnhpr} d_k(H^{\overline{r}}_{\overline{p}}(\mathbb{T}^d), \, L_{\overline{q}}(\mathbb{T}^d)) \le \sum \limits _{m\in \Z_+} d_{\nu_m}(M_{m,\overline{p},\overline{r}}, \, L_{\overline{q}}(\mathbb{T}^d)).
\end{align}

We denote by ${\cal T}(\overline{N}, \, d)_{\overline{p}}$ the unit ball in ${\cal T}(\overline{N}, \, d)$ with respect to the norm of the space $L_{\overline{p}}(\mathbb{T}^d)$.

\begin{Lem}
\label{up_discr} Let $k\in \Z_+$. Then
$$
d_{2k}({\cal T}(\overline{N}, \, d)_{\overline{p}}, \, L_{\overline{q}}(\mathbb{T}^d)) \underset{d}{\lesssim} \prod _{j=1}^d N_j^{1/p_j-1/q_j}d_k(B_{p_1, \, \dots, \, p_d}^{4N_1, \dots, \, 4N_d}, \, l_{q_1, \, \dots, \, q_d}^{4N_1, \dots, \, 4N_d}).
$$
\end{Lem}
The proof is similar to that in the second estimate in Lemma 3.5.8 from \cite{teml_book} (see also the inequality at page 120 after the proof of this lemma).

From \eqref{vp_im}, \eqref{beta_vec_def}, the second estimate in \eqref{fvfrm2rdm}, \eqref{gamma_0_gamma}, \eqref{dnhpr} and Lemma \ref{up_discr} we get
\begin{Lem}
\label{up_discr1} Let $n\in \N$, $\nu_m\in \Z_+$, $m\in \Z_+$, $\sum \limits _{m\in \Z_+} \nu_m\le n$. Then
$$
d_{Cn}(H^{\overline{r}}_{\overline{p}}(\mathbb{T}^d), \, L_{\overline{q}}(\mathbb{T}^d)) \underset{\overline{r},d}{\lesssim}
$$
$$
\lesssim \sum \limits _{m\in \Z_+} 2^{-m \gamma(\overline{p}, \, \overline{q}, \, \overline{r}) \langle \overline{r} \rangle/d} d_{\nu_m} (B_{p_1, \, \dots, \, p_d}^{\lfloor 2^{m\beta_1}\rfloor, \, \dots, \, \lfloor 2^{m\beta_d}\rfloor}, \, l_{q_1, \, \dots, \, q_d}^{\lfloor 2^{m\beta_1}\rfloor, \, \dots, \, \lfloor 2^{m\beta_d}\rfloor}),
$$
where $C=C(d)\in \N$.
\end{Lem}

\subsection{The lower estimate}

First we formulate the Riesz--Thorin theorem for anisotropic norms (see \cite{bed_pan, gal_pan}); more precisely, we write this result in the particular case, which will be further applied. Let $(X_j, \, \mu_j)$, $1\le j\le d$, be finite-measure spaces. By $H$ we denote the linear span of the set of functions $\chi _{E_1\times \dots \times E_d}: X_1\times \dots \times X_d \rightarrow \mathbb{C}$, where $E_j$ are measurable subsets in $X_j$, and $\chi_E$ is the indicator of the set $E$.

\begin{trma}
\label{ri_tor} {\rm (see \cite[Section 7, Theorem 2]{bed_pan}, \cite[Theorem 10.2]{gal_pan}).} Let $0\le \theta \le 1$, $\overline{q}^{0} =(q_1^0, \, \dots, \, q_d^0)$, $\overline{q}^{1} =(q_1^1, \, \dots, \, q_d^1)$, $\overline{q}^{\theta} =(q_1^\theta, \, \dots, \, q_d^\theta)$, $1\le q_j^0\le \infty$, $1\le q_j^1\le \infty$, $\frac{1}{q_j^\theta} = \frac{1-\theta}{q_j^0}+ \frac{\theta}{q_j^1}$. Let $(X, \, \mu)=(X_1\times \dots \times X_d, \, \mu_1\otimes \dots \otimes \mu_d)$, $(Y, \, \nu) = (Y_1\times \dots \times Y_d, \, \nu_1\otimes \dots \otimes \nu_d)$ be products of finite-measure spaces. Let $T$ be a linear operator defined on $H$ with values in the space of measurable functions on $Y_1\times \dots \times Y_d$, and let $$\|Tf\|_{L_{\overline{q}^{0}}(Y, \, \nu)}\le M_0\|f\|_{L_{\overline{q}^{0}}(X, \, \mu)}, \quad \|Tf\|_{L_{\overline{q}^{1}}(Y, \, \nu)}\le M_1\|f\|_{L_{\overline{q}^{1}}(X, \, \mu)}.$$
Then $$\|Tf\|_{L_{\overline{q}^{\theta}}(Y, \, \nu)}\le M_0^{1-\theta}M_1^\theta\|f\|_{L_{\overline{q}^{\theta}}(X, \, \mu)}.$$
\end{trma}

We also need the following property of the Fej\'er kernel \cite[\S 1.2.3]{teml_book}: if $C_1\le mh\le C_2$, then
\begin{align}
\label{2_prop} \sum \limits _{0\le l\le 2\pi/h} {\cal K}_m(x-lh) \le Cm.
\end{align}

\begin{Lem}
\label{low_discr} Let $k\in \Z_+$. Then 
$$
d_k({\cal T}(\overline{N}, \, d)_{\overline{p}}, \, L_{\overline{q}}(\mathbb{T}^d)) \underset{d}{\gtrsim} \prod _{j=1}^d N_j^{1/p_j-1/q_j}d_k(B_{p_1, \, \dots, \, p_d}^{N_1, \dots, \, N_d}, \, l_{q_1, \, \dots, \, q_d}^{N_1, \dots, \, N_d}).
$$
\end{Lem}
\begin{proof}
We repeat arguments from \cite[Lemma 3.5.8]{teml_book}. It suffices to consider the real-valued trigonometric polynomials.

We denote $\nu(\overline{N}) = \{(n_1, \, \dots, \, n_d):\; 1\le n_j\le N_j, \; j=1, \, \dots, \, d\}$. Given $y=(y_{\overline{n}})_{\overline{n}\in \nu(\overline{N})}$, we set
$$
t(y, \, \cdot) = \sum \limits _{\overline{n}\in \nu(\overline{N})}y_{\overline{n}} \varphi(\cdot - z_{\overline{n}}),
$$
where $z_{\overline{n}} =(2\pi n_1/N_1, \, \dots, \, 2\pi n_d/N_d)$, $\varphi(\cdot) = {\cal K}_{\overline{N}-\overline{1}}(\cdot)/{\cal K}_{\overline{N}-\overline{1}}(0)$.

It suffices to prove that
\begin{align}
\label{t_norm_q} \|t(y, \, \cdot)\|_{L_{\overline{q}}(\mathbb{T}^d)} \underset{d}{\lesssim} N_1^{-1/q_1}\dots N_d^{-1/q_d}\|y\|_{l_{q_1, \, \dots, \, q_d}^{N_1, \, \dots, \, N_d}};
\end{align}
the further arguments are similar to that in \cite[Lemma 3.5.8]{teml_book}.

We prove \eqref{t_norm_q} by induction on $d$. The case $d=1$ was obtained in \cite[Lemma 3.5.8]{teml_book}. Now we describe the induction step from $d-1$ to $d$. It suffices to consider the cases $q_d=1$ and $q_d=\infty$; after that we apply Theorem \ref{ri_tor} for $\overline{q}^{0}=(q_1, \, \dots, \, q_{d-1}, \, \infty)$, $\overline{q}^{1}=(q_1, \, \dots, \, q_{d-1}, \, 1)$, $\overline{q}^{\theta}=(q_1, \, \dots, \, q_{d-1}, \, q_d)$.

We set $\overline{m} = (n_1, \, \dots, \, n_{d-1})$, $\overline{M} = (N_1, \, \dots, \, N_{d-1})$, $\xi=(x_1, \, \dots, \, x_{d-1})$, $\zeta_{\overline{m}} = (2\pi n_1/N_1, \, \dots, \, 2\pi n_{d-1}/N_{d-1})$, $z_{d,n_d}= 2\pi n_d/N_d$, $\psi(\xi)= {\cal K}_{\overline{M}-\overline{1}}(\xi)/{\cal K}_{\overline{M}-\overline{1}}(0)$, $\varphi_d(x_d) = {\cal K}_{N_d-1}(x_d)/ {\cal K}_{N_d-1}(0)$.

We fix the point $x_d\in \mathbb{T}$ and denote $b_{\overline{m}} = \sum \limits _{n_d=1}^{N_d} y_{\overline{m},n_d} \varphi_d(x_d-z_{d,n_d})$. By the induction hypothesis,
$$
\|t(y, \, \cdot, \, x_d)\|_{L_{(q_1, \, \dots, \, q_{d-1})}(\mathbb{T}^{d-1})} = \left\| \sum \limits _{\overline{m}}b_{\overline{m}} \psi(\cdot - \zeta_{\overline{m}})\right\|_{L_{(q_1, \, \dots, \, q_{d-1})}(\mathbb{T}^{d-1})} \underset{d}{\lesssim}
$$
$$
\lesssim N_1^{-1/q_1} \dots N_{d-1}^{-1/q_{d-1}} \|(b_{\overline{m}})_{\overline{m}} \| _{l_{q_1, \, \dots, \, q_{d-1}} ^{N_1, \, \dots, \, N_{d-1}}} = $$$$=N_1^{-1/q_1} \dots N_{d-1}^{-1/q_{d-1}} \left\| \Bigl(\sum \limits _{n_d}y_{\overline{m}, n_d} \varphi_d(x_d-z_{d,n_d})\Bigr)_{\overline{m}}\right\| _{l_{q_1, \, \dots, \, q_{d-1}} ^{N_1, \, \dots, \, N_{d-1}}} \le 
$$
$$
\le N_1^{-1/q_1} \dots N_{d-1}^{-1/q_{d-1}} \sum \limits _{n_d} \varphi_d(x_d-z_{d,n_d}) \|(y_{\overline{m}, n_d}) _{\overline{m}} \| _{l_{q_1, \, \dots, \, q_{d-1}} ^{N_1, \, \dots, \, N_{d-1}}}.
$$

For $q_d=\infty$, we have
$$
\|t(y, \, \cdot)\|_{L_{\overline{q}}(\mathbb{T}^d)} \underset{d}{\lesssim} \max _{x_d} N_1^{-1/q_1} \dots N_{d-1}^{-1/q_{d-1}} \sum \limits _{n_d} \varphi_d(x_d-z_{d,n_d}) \|(y_{\overline{m}, n_d}) _{\overline{m}} \| _{l_{q_1, \, \dots, \, q_{d-1}} ^{N_1, \, \dots, \, N_{d-1}}}\le
$$
$$
\le N_1^{-1/q_1} \dots N_{d-1}^{-1/q_{d-1}} \max _{n_d} \|(y_{\overline{m}, n_d}) _{\overline{m}} \| _{l_{q_1, \, \dots, \, q_{d-1}} ^{N_1, \, \dots, \, N_{d-1}}} \max _{x_d}
\sum \limits _{n_d} \varphi_d(x_d-z_{d,n_d}) \stackrel{\eqref{2_prop}}{\lesssim}
$$
$$
\lesssim N_1^{-1/q_1} \dots N_{d-1}^{-1/q_{d-1}} \|(y_{\overline{n}})_{\overline{n}}\| _{l_{q_1, \, \dots, \, q_d}^{N_1, \, \dots, \, N_d}}.
$$

For $q_d=1$,
$$
\|t(y, \, \cdot)\|_{L_{\overline{q}}(\mathbb{T}^d)} \underset{d}{\lesssim} N_1^{-1/q_1} \dots N_{d-1}^{-1/q_{d-1}} \int \limits _{\mathbb{T}}\sum \limits _{n_d} \varphi_d(x_d-z_{d,n_d}) \|(y_{\overline{m}, n_d}) _{\overline{m}} \| _{l_{q_1, \, \dots, \, q_{d-1}} ^{N_1, \, \dots, \, N_{d-1}}}\, dx_d\le
$$
$$
\le N_1^{-1/q_1} \dots N_{d-1}^{-1/q_{d-1}} \sum \limits _{n_d} \|(y_{\overline{m}, n_d}) _{\overline{m}} \| _{l_{q_1, \, \dots, \, q_{d-1}} ^{N_1, \, \dots, \, N_{d-1}}} \max _{n_d} \int \limits _{\mathbb{T}}\varphi_d(x_d-z_{d,n_d})\, dx_d \lesssim
$$
$$
\lesssim N_1^{-1/q_1} \dots N_{d-1}^{-1/q_{d-1}} N_d^{-1} \|(y_{\overline{n}})_{\overline{n}}\| _{l_{q_1, \, \dots, \, q_d}^{N_1, \, \dots, \, N_d}}.
$$
This completes the proof.
\end{proof}

Let $m\in \Z_+$, $\overline{\beta} = (\beta_1, \, \dots, \, \beta_d)$, and let $\lfloor 2^{\overline{\beta}m}\rfloor$ be defined by formulas \eqref{beta_vec_def} and \eqref{2beta_vec_def}, 
\begin{align}
\label{til_m_prm}
\tilde M_{\overline{p},\overline{r},m} = 2^{-m\langle \overline{r}\rangle/d}{\cal T}(\lfloor 2^{\overline{\beta} m}\rfloor,d)_{\overline{p}}. 
\end{align}
Applying Lemma \ref{low_discr} for $\overline{N} = \lfloor 2^{\overline{\beta}m} \rfloor$ and taking into account \eqref{beta_vec_def}, \eqref{gamma_0_gamma}, we get
\begin{align}
\label{m_pr_low} 
d_k(\tilde M_{\overline{p},\overline{r},m}, \, L_{\overline{q}}(\mathbb{T}^d)) \underset{d}{\gtrsim} 2^{-m\gamma(\overline{p},\overline{q},\overline{r})\langle \overline{r}\rangle/d}d_k(B_{p_1, \, \dots, \, p_d}^{\lfloor 2^{\beta_1m}\rfloor, \dots, \, \lfloor 2^{\beta_dm}\rfloor}, \, l_{q_1, \, \dots, \, q_d}^{\lfloor 2^{\beta_1m}\rfloor, \dots, \, \lfloor 2^{\beta_dm}\rfloor}).
\end{align}

Let $r\ge 0$, $\alpha\in \R$, $n\in \N$. We set for $x\in \R$
$$
{\cal V}^r_n(x, \, \alpha) = 1+2 \sum \limits _{k=1}^n k^r \cos (kx+\alpha \pi/2) + 2 \sum \limits _{k=n+1}^{2n-1}k^r \left(1-\frac{k-n}{n}\right)\cos (kx+\alpha \pi/2).
$$
For $t\in {\cal T}(n, \, 1)$, we set $D^r_\alpha t = t * {\cal V}^r_n(\cdot, \, \alpha)$, $I^r_\alpha t= t* F_r(\cdot, \, \alpha)$, where $F_r(\cdot, \, \alpha)$ is the Bernoulli kernel (see \eqref{frxa}). Then (see \cite[formula (1.4.8)]{teml_book})
\begin{align}
\label{dif_int} D^r_\alpha I^r_\alpha t= I^r_\alpha D^r_\alpha t = t, \quad D^0_0t = t.
\end{align}

Let now $\overline{r}=(r_1, \, \dots, \, r_d)\in \R_+^d$, $\overline{\alpha} = (\alpha_1, \, \dots, \, \alpha_d)\in \R^d$, $\overline{N}=(N_1,\, \dots, \, N_d)\in \N^d$, $x=(x_1, \, \dots, \, x_d) \in \mathbb{T}^d$. We set $${\cal V}^{\overline{r}}_{\overline{N}}(x, \, \overline{\alpha}) = \prod _{j=1}^d {\cal V}^{r_j}_{N_j}(x_j, \, \alpha_j), \quad D^{\overline{r}}_{\overline{\alpha}} t = t* {\cal V}^{\overline{r}}_{\overline{N}}(\cdot, \, \alpha), \quad t \in {\cal T}(\overline{N}, \, d).$$

\begin{trma}
\label{bern_ineq} {\rm (the Bernstein inequality; see \cite[Theorem 3.3.1]{teml_book})} Let $r_j\ge 0$, $\alpha_j\in \R$ $(1\le j\le d)$, and let $\alpha_j=0$ for all $j$ such that $r_j=0$. Then, for all $\overline{N}\in \N^d$, $\overline{p}\in [1, \, \infty]^d$, $t\in {\cal T}(\overline{N}, \, d)$,
$$
\|D^{\overline{r}}_{\overline{\alpha}}t\|_{L_{\overline{p}}(\mathbb{T}^d)} \underset{\overline{r}}{\lesssim} \|t\|_{L_{\overline{p}}(\mathbb{T}^d)} \prod_{j=1}^d N_j^{r_j}.
$$
\end{trma}

Let $t\in \tilde M_{\overline{p},\overline{r},m}$ (see \eqref{til_m_prm}), $\overline{N} = \lfloor 2^{\overline{\beta}m}\rfloor$. For $1\le j\le d$ we set $\overline{r}_j=(0, \, \dots, \, 0, \, r_j, \, 0, \, \dots, \, 0)$, $\overline{\alpha}_j=(0, \, \dots, \, 0, \, \alpha_j, \, 0, \, \dots, \, 0)$ ($r_j$ and $\alpha_j$ are on $j$th position), $\varphi_j = D^{\overline{r}_j}_{\overline{\alpha}_j} t$. By \eqref{dif_int},
\begin{align}
\label{t_def_frj}
t(x_1, \, \dots, \, x_d) = \frac{1}{2\pi} \int \limits _{\mathbb{T}} \varphi_j(x_1, \, \dots, \, x_{j-1}, \, y, \, x_{j+1}, \, \dots, \, x_d) F_{r_j}(x_j-y, \, \alpha_j)\, dy;
\end{align}
from Theorem \ref{bern_ineq} we get
$$
\|\varphi_j\|_{L_{\overline{p}}(\mathbb{T}^d)} \underset{\overline{r}}{\lesssim} \|t\|_{L_{\overline{p}}(\mathbb{T}^d)} N_j^{r_j} \stackrel{\eqref{til_m_prm}}{\le} 2^{-m\langle \overline{r}\rangle/d} 2^{m\beta_j r_j} \stackrel{\eqref{beta_vec_def}}{=} 1, \quad 1\le j\le d.
$$
Hence (see \eqref{t_def_frj} and the definition of $W^{\overline{r}}_{\overline{p},\overline{\alpha}}(\mathbb{T}^d)$) there is $C(\overline{r})>0$ such that
$$
\tilde M_{\overline{p},\overline{r},m} \subset C(\overline{r}) W^{\overline{r}}_{\overline{p},\overline{\alpha}}(\mathbb{T}^d).
$$
This together with \eqref{m_pr_low} implies
\begin{Lem}
\label{low_discr1} Let $k\in \Z_+$, $m\in \Z_+$. Then
$$
d_k(W^{\overline{r}}_{\overline{p},\overline{\alpha}}(\mathbb{T}^d), \, L_{\overline{q}}(\mathbb{T}^d)) \underset{\overline{r},d}{\gtrsim} 2^{-m\gamma(\overline{p},\overline{q},\overline{r})\langle \overline{r}\rangle/d}d_k(B_{p_1, \, \dots, \, p_d}^{\lfloor 2^{\beta_1m}\rfloor, \dots, \, \lfloor 2^{\beta_dm}\rfloor}, \, l_{q_1, \, \dots, \, q_d}^{\lfloor 2^{\beta_1m}\rfloor, \dots, \, \lfloor 2^{\beta_dm}\rfloor}).
$$
\end{Lem}

\section{Proof of Theorems \ref{main1} and \ref{main2}}

In order to prove the upper (lower) estimate, we use Lemma \ref{up_discr1} (respectively, Lemma \ref{low_discr1}). Then we apply Theorem \ref{fim_dim_poper} and Proposition \ref{fim_dim_poper2}.

We prove Theorem \ref{main1} (the proof of Theorem \ref{main2} is similar and easier).

We define $\overline{\beta}$ by \eqref{beta_vec_def}; given $m\in \Z_+$, we set $\overline{N} = (N_1, \, \dots, \, N_d)=\lfloor 2^{\overline{\beta}m}\rfloor$ according to \eqref{2beta_vec_def}; the number $\gamma(\overline{p},\overline{q},\overline{r})$ is defined by  \eqref{gamma_0_gamma}.

Let $n\le \frac{N_1\dots N_d}{2}$. We obtain the order estimates of
$$
2^{-m\gamma(\overline{p},\overline{q},\overline{r})\langle \overline{r}\rangle/d}d_n(B_{p_1, \, \dots, \, p_d}^{\lfloor 2^{\beta_1m}\rfloor, \dots, \, \lfloor 2^{\beta_dm}\rfloor}, \, l_{q_1, \, \dots, \, q_d}^{\lfloor 2^{\beta_1m}\rfloor, \dots, \, \lfloor 2^{\beta_dm}\rfloor}).
$$
To this end, we use Theorem \ref{fim_dim_poper}.

For $n\le N_{\sigma(1)}\dots N_{\sigma(\mu)} N _{\sigma(\mu+1)}^{2/q_{\sigma(\mu+1)}} \dots N_{\sigma(d)} ^{2/q_{\sigma(d)}}$, we have (see \eqref{om_pq}, \eqref{upor}, \eqref{fin_dim_est1}, \eqref{dn11})
\begin{align}
\label{d_est0h}
\begin{array}{c}
2^{-m\gamma(\overline{p},\overline{q},\overline{r})\langle \overline{r}\rangle/d}d_n(B_{p_1, \, \dots, \, p_d}^{\lfloor 2^{\beta_1m}\rfloor, \dots, \, \lfloor 2^{\beta_dm}\rfloor}, \, l_{q_1, \, \dots, \, q_d}^{\lfloor 2^{\beta_1m}\rfloor, \dots, \, \lfloor 2^{\beta_dm}\rfloor}) \underset{\overline{q}}{\asymp} \\ \asymp 2^{-m(1+(d-\mu)/\langle \overline{q}\circ \overline{r}\rangle_{I(\mu+1, \, d)}-(d-\mu)/\langle \overline{p}\circ \overline{r}\rangle_{I(\mu+1, \, d)})\langle \overline{r}\rangle/d}.
\end{array}
\end{align}

For $t\in \{\mu+1, \, \mu+2, \, \dots, \, \nu\}$, by \eqref{mu_nu_def}, we get $\omega_{p_{\sigma(t)}, q_{\sigma(t)}} \in (0, \, 1)$; this together with \eqref{om_pq} implies that $q_{\sigma(t)}>2$, $p_{\sigma(t)}\in (2, \, q_{\sigma(t)})$.

Let $t\in \{\mu+1, \, \mu+2, \, \dots, \, \nu\}$, $$N_{\sigma(1)} \dots N_{\sigma(t-1)}N_{\sigma(t)}^{2/q_{\sigma(t)}} \dots N_{\sigma(d)}^{2/q_{\sigma(d)}} < n \le N_{\sigma(1)} \dots N_{\sigma(t)}N_{\sigma(t+1)}^{2/q_{\sigma(t+1)}} \dots N_{\sigma(d)}^{2/q_{\sigma(d)}}.$$
Then (see \eqref{fin_dim_est1}, \eqref{ks12qt}, \eqref{dn22})
\begin{align}
\label{d_est1h}
\begin{array}{c}
2^{-m\gamma(\overline{p},\overline{q},\overline{r})\langle \overline{r}\rangle/d}d_n(B_{p_1, \, \dots, \, p_d}^{\lfloor 2^{\beta_1m}\rfloor, \dots, \, \lfloor 2^{\beta_dm}\rfloor}, \, l_{q_1, \, \dots, \, q_d}^{\lfloor 2^{\beta_1m}\rfloor, \dots, \, \lfloor 2^{\beta_dm}\rfloor}) \underset{\overline{q}}{\asymp} 
\\
\asymp 2^{-m\frac{\langle \overline{r}\rangle}{d}(1+(d-t+1)(1/\langle \overline{r}\circ \overline{q}\rangle_{I(t, \, d)} - 1/\langle \overline{r}\circ \overline{p}\rangle_{I(t, \, d)}))}  \Bigl( n^{-\frac 12} 2^{m\frac{\langle \overline{r}\rangle}{d}\left(\frac{t-1}{2\langle \overline{r}\rangle _{I(1, \, t-1)}}+\frac{d-t+1}{\langle \overline{r}\circ \overline{q}\rangle_{I(t, \, d)}}\right)}\Bigr)^{\frac{1/p_{\sigma(t)}-1/q_{\sigma(t)}}{1/2-1/q_{\sigma(t)}}}.
\end{array}
\end{align}

If $\nu<d$ and $n> N_{\sigma(1)}\dots N_{\sigma(\nu)}N_{\sigma(\nu+1)}^{2/q_{\sigma(\nu+1)}} \dots N_{\sigma(d)}^{2/q_{\sigma(d)}}$, then (see \eqref{fin_dim_est1}, \eqref{dn33})
\begin{align}
\label{d_est2h}
\begin{array}{c}
2^{-m\gamma(\overline{p},\overline{q},\overline{r})\langle \overline{r}\rangle/d}d_n(B_{p_1, \, \dots, \, p_d}^{\lfloor 2^{\beta_1m}\rfloor, \dots, \, \lfloor 2^{\beta_dm}\rfloor}, \, l_{q_1, \, \dots, \, q_d}^{\lfloor 2^{\beta_1m}\rfloor, \dots, \, \lfloor 2^{\beta_dm}\rfloor}) \underset{\overline{q}}{\asymp} 
\\
\asymp 2^{-m\frac{\langle \overline{r}\rangle}{d}(1-(d-\nu)/\langle \overline{p}\circ \overline{r}\rangle_{I(\nu+1, \, d)}-\nu/\langle 2\overline{r}\rangle_{I(1, \, \nu)})} n^{-1/2}.
\end{array}
\end{align}

Let $T=\{\mu, \, \mu+1, \, \dots, \, \nu\}$. For $t\in T$ we define the number $s_t$ by the equation
$$
s_t\cdot \frac{\langle \overline{r}\rangle}{d}\left( \sum \limits _{j=1}^t \frac{1}{r_{\sigma(j)}} + \sum \limits _{j=t+1}^d \frac{2}{q_{\sigma(j)}r_{\sigma(j)}}\right)=1;
$$
then $s_t$ strictly decreases in $t\in T$ and
$$
s_t = \frac{1}{\frac{t}{d}\cdot \frac{\langle \overline{r}\rangle}{\langle \overline{r}\rangle_{I(1, \, t)}}+ 2\frac{d-t}{d}\cdot \frac{\langle \overline{r}\rangle}{\langle \overline{r}\circ \overline{q}\rangle_{I(t+1, \, d)}}}.
$$
If $\nu=d$ or $q_{\sigma(j)}=2$ for all $j\in \{\nu + 1, \, \dots, \, d\}$, then $s_\nu=1$.

If $\nu<d$ and there is $j\in \{\nu + 1, \, \dots, \, d\}$ such that $q_{\sigma(j)}>2$, then $s_\nu>1$; in this case, we set $s_d=1$.

Let $J$ be the set from the formulation of Theorem \ref{main1}. We define the function $h:[1, \, s_\mu] \rightarrow \R$ by
$$
h = \max _{t+1\in J} h_t,
$$
where
$$
h_{\mu-1}(s) = \frac{\langle \overline{r} \rangle}{d}\Bigl( 1+ \frac{d-\mu}{\langle \overline{q}\circ \overline{r} \rangle_{I(\mu+1, \, d)}}-\frac{d-\mu}{\langle \overline{p}\circ \overline{r} \rangle_{I(\mu+1, \, d)}}\Bigr)s,
$$
$$
h_t(s) = \frac{\langle \overline{r} \rangle}{d} \Bigl( 1+ \frac{d-t}{\langle \overline{q}\circ \overline{r} \rangle_{I(t+1, \, d)}}-\frac{d-t}{\langle \overline{p}\circ \overline{r} \rangle_{I(t+1, \, d)}}\Bigr)s +
$$
$$
+ \frac{1/p_{\sigma(t+1)}-1/q_{\sigma(t+1)}}{1/2-1/q_{\sigma(t+1)}}\Bigl(\frac 12 - s\cdot \frac{\langle \overline{r} \rangle}{d}\Bigl(\frac{t}{\langle 2\overline{r} \rangle _{I(1, \, t)}} + \frac{d-t}{\langle \overline{q} \circ \overline{r} \rangle _{I(t+1, \, d)}}\Bigr)\Bigr)
$$
for $\mu\le t \le \nu-1$; if $\nu<d$ and there is $j\in \{\nu+1, \, \dots, \, d\}$ such that $q_{\sigma(j)}>2$, we set
$$
h_{d-1}(s) = \frac{\langle \overline{r} \rangle}{d} \Bigl(1-\frac{\nu}{\langle 2\overline{r}\rangle_{I(1,\, \nu)}} - \frac{d-\nu}{\langle \overline{p} \circ \overline{r} \rangle _{I(\nu+1, \, d)}} \Bigr)s + \frac 12.
$$
Notice that the right-hand side of \eqref{d_est0h} is equal to $n^{-h_{\mu-1}(m/\log_2 n)}$, the right-hand side of \eqref{d_est1h} is equal to $n^{-h_{t-1}(m/\log_2 n)}$ ($\mu\le t\le \nu$), and the right-hand side of \eqref{d_est2h} is equal to $n^{-h_{d-1}(m/\log_2 n)}$ for $\nu<d$.

Now we apply Lemmas \ref{up_discr1} and \ref{low_discr1}, taking into account \eqref{emb_w_h}, \eqref{emb_cond}, \eqref{d_est0h}, and argue similarly as in \cite[proof of Theorem 1]{vas_int_sob}; we get that if $h$ has the unique minimum point on $[1, \, s_\mu]$ (we denote it by $s_*$), then
$$
n^{-h(s_*)}\underset{\overline{p},\overline{q},\overline{r},d}{\lesssim} d_n(W^{\overline{r}}_{\overline{p}, \overline{\alpha}}(\mathbb{T}^d), \, L_{\overline{q}}(\mathbb{T}^d)) 
\underset{\overline{r},d}{\lesssim}
d_n(H^{\overline{r}}_{\overline{p}}(\mathbb{T}^d), \, L_{\overline{q}}(\mathbb{T}^d)) 
\underset{\overline{p},\overline{q},\overline{r},d}{\lesssim} n^{-h(s_*)}.
$$

Let $\{\theta_t\}_{t\in J}$ be the numbers from Theorem \ref{main1}.

If $s_\mu=1$, then by the strict monotonicity of $\{s_t\}_{t\in J}$ we get that $T=J=\{\mu\}$; in the case $\mu<d$ we have $q_{\sigma(j)}=2$, $\mu+1\le j\le d$. Hence, $h=h_{\mu-1}$ and $h(1)=\theta_\mu$.

Let $s_\mu>1$. Notice that the following equalities hold: $h_{\mu-1}(s_\mu)=\theta_\mu$; for $\mu\le t\le \nu-1$ we have $h_t(s_t)=\theta_t$; for $\mu+1\le t\le \nu$ we have $h_{t-1}(s_t)=\theta_t$; if $\nu<d$ and $q_{\sigma(j)}>2$ for some $j\in \{\nu+1, \, \dots, \, d\}$, then $h_{d-1}(s_\nu) = \theta _{\nu}$, $h_{d-1}(1) = \theta_d$. In addition, from \eqref{upor} and \eqref{mu_nu_def} it follows that $h_t'\le h_{t-1}'$, $\mu\le t\le \nu$, and $h'_{d-1}\le h'_{\nu-1}$ in the case $s_\nu>1$. Hence, if $\mu\le t\le \nu-1$, then $h(s) = h_t(s)$ for $s_{t+1}\le s\le s_t$; $h(s) = h_{d-1}(s)$ for $s_\nu>1$, $1\le s\le s_\nu$. Therefore, the minimum of $h$ is attained at the point $s_t$ for some $t\in J$, and $\min _{[1,\, s_\mu]}h=\theta_t$. This completes the proof.

\begin{Rem}
If the condition \eqref{emb_cond} in Theorem {\rm \ref{main1}} {\rm (}respectively, \eqref{theta_posit} in Theorem {\rm \ref{main2})} fails, we similarly obtain that
$$
d_n(W^{\overline{r}}_{\overline{p}, \overline{\alpha}}(\mathbb{T}^d), \, L_{\overline{q}}(\mathbb{T}^d))\underset{\overline{p},\overline{q},\overline{r},d}{\gtrsim} 1, \quad d_n(H^{\overline{r}}_{\overline{p}}(\mathbb{T}^d), \, L_{\overline{q}}(\mathbb{T}^d))\underset{\overline{p},\overline{q},\overline{r},d}{\gtrsim} 1;
$$
i.e., the embedding is not compact.
\end{Rem}

It was mentioned above that in \cite[Chapter III, Section 15.2]{besov_iljin_nik}, \cite[Chapter \S 2]{itogi_nt}, \cite{galeev_emb, galeev1, galeev2, galeev85, galeev87, galeev4} the Sobolev classes were defined in a different way. Each distribution $f\in {\cal S}'(\mathbb{T}^d)$ has the Fourier series: $f = \sum \limits _{\overline{k}\in \Z^d} c_{\overline{k}}(f) e^{i(\overline{k}, \, \cdot)}$; it converges in the topology of ${\cal S}'(\mathbb{T}^d)$; here $(\cdot, \, \cdot)$ is the standard inner product on $\R^d$; the test-functions are infinitely smooth functions on $\mathbb{T}^d$. We set $$\mathaccent'27 \Z^d = \{(k_1, \, k_2, \, \dots, \, k_d)\in \Z^d:\; k_1k_2\dots k_d\ne 0\},$$
$$
\mathaccent'27{\cal S}'(\mathbb{T}^d) = \Bigl\{ f\in {\cal S}'(\mathbb{T}^d):\; f=\sum \limits _{\overline{k}\in \mathaccent'27\Z^d} c_{\overline{k}}(f) e^{i(\overline{k}, \, \cdot)}\Bigr\}.
$$
Let $r_j>0$, $1\le j\le d$. The partial Weyl derivative of order $r_j$ with respect to $x_j$ of the distribution $f\in \mathaccent'27{\cal S}'(\mathbb{T}^d)$ is defined by
$$
\partial_j^{r_j}f:=\frac{\partial^{r_j}f}{\partial x_j^{r_j}} := \sum \limits _{\overline{k}\in \mathaccent'27\Z^d} c_{\overline{k}}(f) (ik_j)^{r_j} e^{i(\overline{k}, \, \cdot)},
$$
where $(ik_j)^{r_j}=|k_j|^{r_j}e^{{\rm sgn}\,k_j\cdot i\pi r_j/2}$. 

Let $1\le p_j\le \infty$, $r_j>0$, $1\le j\le d$, $\overline{p}=(p_1, \, \dots, \, p_d)$, $\overline{r} = (r_1, \, \dots, \, r_d)$. We set
$$
\tilde{W}^{\overline{r}}_{\overline{p}}(\mathbb{T}^d) = \{f\in \mathaccent'27{\cal S}'(\mathbb{T}^d):\; \|\partial _j^{r_j}f\|_{L_{\overline{p}}(\mathbb{T}^d)}\le 1, \; 1\le j\le d\}.
$$

It turns out that 
\begin{align}
\label{twrp}
\tilde{W}^{\overline{r}}_{\overline{p}}(\mathbb{T}^d) \subset W^{\overline{r}}_{\overline{p},\overline{r}}(\mathbb{T}^d). 
\end{align}
In order to check this, it suffices to prove the equality $I_{j,r_j} \partial_j^{r_j}f = f$, $j=1, \, \dots, \, d$, $f\in \tilde{W}^{\overline{r}}_{\overline{p}}(\mathbb{T}^d)$, where
$$
I_{j,r_j}\varphi(x_1, \, \dots, \, x_d) = \frac{1}{2\pi} \int \limits _{\mathbb{T}} \varphi(x_1, \, \dots, \, x_{j-1}, \, y, \, x_{j+1}, \, \dots, \, x_d) F_{r_j}(x_j-y, \, r_j)\, dy
$$
(see \eqref{oooo}). For trigonometric polynomials it can be checked directly. For $f\in \tilde{W}^{\overline{r}}_{\overline{p}}(\mathbb{T}^d)$ the equality can be proved by passing to the limit; here we use the following facts: 1) the operator $I_{j,r_j}$ is continuous on $L_{\overline{p}}(\mathbb{T}^d)$; 2) $V_{(N, \, \dots, \, N)}f \underset{N\to \infty}{\to} f$ in $L_{\overline{p}}(\mathbb{T}^d)$ and in $\mathaccent'27{\cal S}'(\mathbb{T}^d)$; 3) $\partial _j^{r_j}(V_{\overline{N}}f) = V_{\overline{N}}(\partial _j^{r_j}f)$.

From \eqref{twrp} it follows that $d_n(\tilde{W}^{\overline{r}}_{\overline{p}}(\mathbb{T}^d), \, L_{\overline{q}}(\mathbb{T}^d))\le d_n(W^{\overline{r}}_{\overline{p},\overline{r}}(\mathbb{T}^d), \, L_{\overline{q}}(\mathbb{T}^d))$; this together with Theorems \ref{main1} and \ref{main2} yields the upper estimates. The estimate from below has the same order; the proof is similar to that for $W^{\overline{r}}_{\overline{p},\overline{\alpha}}(\mathbb{T}^d)$; here in the definition of $\tilde M_{\overline{p},\overline{r},m}$ in \eqref{til_m_prm} we take, instead of ${\cal T}(\lfloor 2^{\overline{\beta} m}\rfloor,d)$, the polynomials with harmonics in $\prod_{j=1}^d\{1, \, 2, \, 3, \, \dots, \, 2\lfloor 2^{m\beta_jr_j}\rfloor +1\}$.

\vskip 0.3cm

In conclusion, the author expresses her sincere gratitude to V. N. Temlyakov and
G. A. Akishev for useful discussion and references.

\end{document}